\documentclass[11pt]{amsart} 
\usepackage{amssymb,amsmath,latexsym,enumerate,graphicx,bbm,mathptmx,microtype,cite}
\allowdisplaybreaks
\usepackage[hmargin=1.25in, vmargin=1.5in]{geometry} 
\usepackage[shortlabels]{enumitem}

\usepackage{float}

\usepackage{amsthm, pict2e, mathtools }
\usepackage{xfrac}    
\usepackage{faktor} 
\usepackage{easybmat} 

\usepackage{xcolor}
\newlength\mylen
\usepackage{array} 
\newcolumntype{C}{>>{\hfil$}p{\mylen}<{$\hfil}}
\usepackage{tikz}
\usetikzlibrary{graphs,graphs.standard}
\usepackage{tikz-cd}
\usepackage{caption}
\usepackage{subcaption}
\usepackage{color}
\usepackage{scrlayer-scrpage}
\usepackage{siunitx}
\usepackage{lastpage}
\usepackage{comment}
\usetikzlibrary{arrows,intersections,patterns}
\usetikzlibrary{decorations.markings}
\usepackage[%
hypertexnames=false
]{hyperref} 
\hypersetup{
    colorlinks = true,
    urlcolor = blue,
    linkcolor = blue,
    citecolor = blue
}
\usepackage{cleveref}
\usepackage{autonum}

\newtheorem{theorem}{Theorem}[section] 

\makeatletter 
\newtheorem*{rep@theorem}{\rep@title}\newcommand{\newreptheorem}[2]{%
\newenvironment{rep#1}[1]{%
\def\rep@title{\bf #2 \ref{##1}}%
\begin{rep@theorem}}%
{\end{rep@theorem}}}
\makeatother

\newreptheorem{theorem}{Theorem}
\newtheorem{cor}[theorem]{Corollary}
\newtheorem{prop}[theorem]{Proposition}
\newtheorem{lem}[theorem]{Lemma}

\newtheorem*{exam}{Example}
\newenvironment{example}{\begin{exam}\rm}{\end{exam}}

\theoremstyle{definition}
\newtheorem{remark}[theorem]{Remark}

\newcommand\ehr{\operatorname{ehr}}

\newcommand\tc{\operatorname{tc}} 
 
\newcommand\lc{\operatorname{lc}} 
\newcommand\pc{\operatorname{pc}} 
\newcommand\tr{\operatorname{tr}} 
\newcommand\lt{\operatorname{lt}} 
\newcommand\pt{\operatorname{pt}} 
\newcommand\ct{\operatorname{ct}} 
\newcommand\cl{\operatorname{cl}} 
\newcommand\cp{\operatorname{cp}} 
 
\newcommand\lspan{\operatorname{span}} 
\newcommand\mplc{\operatorname{mplc}} 

\newcommand\ctc{\operatorname{ctc}} 
 
\newcommand\clc{\operatorname{clc}} 
\newcommand\cpc{\operatorname{cpc}}

\newcommand\ZZ{\mathbb{Z}}
\newcommand\QQ{\mathbb{Q}}
\newcommand\RR{\mathbb{R}}

\newcommand\FF{\mathbb{F}}

\newcommand\cH{\mathcal{H}}
\newcommand\cI{\mathcal{I}}

\newcommand\Def[1]{{\bf #1}}

\newcommand\Zono{\mathcal{Z}}
\newcommand\Pol{\mathcal{P}}

\newcommand{\zono}{\operatorname{\mathcal{Z}(M(G))}}
\newcommand{\cozono}{\operatorname{\mathcal{Z}(M(G)^{\triangle})}}

\newcommand{\tocyc}{\operatorname{\mathcal{T}}}

\newcommand\ba{\mathbf{a}}
\newcommand\bb{\mathbf{b}}

\newcommand\be{\mathbf{e}}

\newcommand\bu{\mathbf{u}}
\newcommand\bv{\mathbf{v}}
\newcommand\bw{\mathbf{w}}
\newcommand\bx{\mathbf{x}}
\newcommand\by{\mathbf{y}}

\newcommand\bA{\mathbf{A}}
\newcommand\bB{\mathbf{B}}
\newcommand\bD{\mathbf{D}}
\newcommand\bF{\mathbf{F}}
\newcommand\bI{\mathbf{I}}
\newcommand\bJ{\mathbf{J}}
\newcommand\bS{\mathbf{S}}
\newcommand\bR{\mathbf{R}}
\newcommand\bT{\mathbf{T}}

\newcommand\bW{\mathbf{W}}
\newcommand\bzero{\mathbf{0}}
\newcommand\nn{n} 
\newcommand\ned{m} 

\DeclareMathOperator{\kernel}{kern}

\newcommand{\matroid}{M}
\newcommand{\dualmatroid}{M^\triangle}
\newcommand{\graph}{G}
\newcommand{\sgraph}{\Sigma} 
\newcommand{\ugraph}{\Gamma} 
\newcommand{\subsgraph}[1]{\sgraph({#1})} 
\newcommand{\subugraph}[1]{\ugraph({#1})} 

\newcommand{\g}{g} 

\newcommand\compl[1]{\overline{#1}} 
\newcommand\lfrac[2]{\Large \sfrac{#1}{#2}} 

\usetikzlibrary{arrows.meta}
\usepackage{tikz-cd}

\usepackage[ruled, linesnumbered, noend]{algorithm2e}

\DeclareMathOperator{\rank}{rank}

\newcommand\rspan{\mathrm{span}_\RR}
\newcommand\zspan{\mathrm{span}_\ZZ}

\newcommand{\matt}[1]{{\color{blue} \sf  Matt: [#1]}}
\newcommand{\sophie}[1]{{\color{teal} \sf Sophie: [#1]}}

\newcommand{\nexttodo}[1]{{\color{red} \sf Next To Do: [#1]}}

\renewcommand\emptyset{\varnothing}

\newcommand\commentout[1]{}

\title{Acyclotopes, Tocyclotopes, and their Ehrhart Polynomials}

\author{Eleon Bach}
\address{Institut f\"{u}r Mathematik\\Technische Universit\"{a}t M\"{u}nchen }
\email{eleon.bach@tum.de}

\author{Matthias Beck}
\address{Department of Mathematics\\
         San Francisco State University
         }
\email{mattbeck@sfsu.edu}

\author{Sophie Rehberg}
\address{Département de Mathématiques \\ Université du Québec à Montréal}
\email{rehberg.sophie@uqam.ca}

\begin{document}



\thanks{We thank 
Federico Ardila, 
Gennadiy Averkov,
Florian Frick, 
Christian Haase, 
Martin Henk,
Raman Sanyal,
and an anonymous referee for helpful comments.
The third author would like to thank the Berlin Mathematical School for funding research stays during which part of this work was done,
funded by the Deutsche Forschungsgemeinschaft (DFG, German Research
Foundation) under Germany's Excellence Strategy-–-the Berlin Mathematics
Research Center MATH+ (EXC-2046/1, project ID: 390685689).
}

\date{9 June 2026}

\begin{abstract}  
There is a well-established dictionary between zonotopes, hyperplane arrangements, and their (oriented) matroids. 
Arguably one of the most famous examples is the class of graphical zonotopes, also called
acyclotopes, which encode Minkowski summands of the type-A permutahedron.
Stanley (1991) gave a general interpretation of the coefficients of the Ehrhart polynomial
(integer-point counting function for a polytope) of a zonotope via linearly independent
subsets of its generators. Applying this to the graphical case shows that Ehrhart coefficients count  forests of the graph of fixed sizes.
Our first goal is to extend and popularize this story to other root systems, which on the
combinatorial side is encoded by signed graphs analogously to the work by Greene and Zaslavsky (1983).
We compute the Ehrhart polynomial of the acyclotope in the signed case.
Our second goal is to translate matroid duality concepts not only to hyperplanes but also to zonotopes in such a way that the Ehrhart polynomial is uniquely defined.
In the case of (signed) graphs we construct tocyclotopes and compute their Ehrhart polynomials.
Applying the same duality construction to a general integral matrix gives rise to a lattice
Gale zonotope, whose face structure was studied by McMullen (1971) and whose duality concepts
give a special instance of D'Adderio--Moci's arithmetic matroids (2013). We describe its Ehrhart
polynomials in terms of the given matrix.
\end{abstract}

\maketitle


\section{Introduction}\label{sec:intro}

For every (oriented) graph $\graph =(V,E)$, we consider the following triple: First,  there exists an associated matroid $\matroid(\graph) = (E, \cI)$ whose ground set is $E$ and whose set of independent sets $\cI$ is the set of trees in $\graph$.
Second, for every graph there exists an induced (central) hyperplane arrangement whose normal vectors are the columns of the incidence matrix $\bA_\graph \coloneqq [ \be_j - \be_k : \, jk \in E ] \in \RR^{\nn \times \ned}$ of $\graph$ where we denote by $\be_j$ the standard basis vector in $\RR^\nn$ corresponding to the node $j$.
Here $\nn = |V|$, $\ned = |E|$.
We call this hyperplane arrangement the \Def{graphic arrangement} $\cH(\bA_\graph)$. 
Third,
there is a well-developed dictionary between the \Def{zonotope} generated by $\bA_\graph$,
\[
  \Zono(\bA_\graph) := \bA_\graph \, [0,1]^\ned ,
\]
and the graphic hyperplane arrangement. 
Greene and Zaslavsky~\cite{ZasGre} showed that the vertices of $\Zono(\bA_\graph)$  and equivalently, the regions of $\cH(\bA_\graph)$ are in one-to-one correspondence with the acyclic orientations\footnote{
An orientation is called \Def{acyclic} if it does not contain any coherently oriented cycles.
}
of $\graph$, and they gave analogous interpretations for all faces of~$\Zono(\bA_\graph)$. Zaslavsky
thus coined the charming term \Def{acyclotope} for $\Zono(\bA_\graph)$, one that we would like to revitalize.
Zaslavsky \cite[Section 4]{zaslavskysignedcoloring} introduced the acyclotope, in fact, in the more general setting of signed graphs, which is one that we will turn to later. 

In fact, the dictionary described for graphs above can be phrased for general matrices $\bA\in\RR^{\nn\times\ned}$, that is, for general representable matroids (over $\RR$), hyperplane arrangements $\cH(\bA)$, and zonotopes $\Zono(\bA)$.

Zonotopes admit additional arithmetic
data compared to the associated hyperplane arrangements. Each $\Zono(\bA)$ comes with a natural tiling into
parallelepipeds~\cite{mcmullernzonotopes,shephardzonotopes}, whose (relative) volumes
encode important arithmetic data of $\Zono(\bA)$.
When $\bA \in \ZZ^{n \times m}$, i.e., $\Zono(\bA)$ is a \Def{lattice zonotope},
they are most easily packaged into the \Def{Ehrhart polynomial} of $\Zono(\bA)$.
For a general lattice polytope $\Pol\subset\RR^n$, i.e., the convex hull of
finitely many points in $\ZZ^n$,  the Ehrhart polynomial is defined as 
\begin{equation}
  \ehr_{\Pol}(t) \coloneqq \left| t \Pol \cap \ZZ^n \right| \,,
\end{equation}
for positive integers $t$.  There is a rich theory and many applications to these
counting functions, including the topics discussed in this paper; see, e.g., \cite{reciprbook}.
For lattice zonotopes Stanley \cite{MR1116376} proved the following result.
\begin{theorem}[Stanley]\label{thm:stanley}
Let $\mathcal{Z}$ be a zonotope generated by the integer vectors $\ba_1, \dots, \ba_m \in
\ZZ^n$. Then the Ehrhart polynomial of $\mathcal{Z}$ equals 
\begin{align}
    \ehr_{\mathcal{Z}}(t) = \sum_\bF \g(\bF) \, t^{|\bF|}
\end{align}
where $\bF$ ranges over all linearly independent subsets of $\{ \ba_1, \dots, \ba_m \}$ and $\g(\bF)$ is 
the (relative) volume of the parallelepiped spanned by the vectors in $\bF$, or equivalently
it is the greatest common divisor of all minors of size $|\bF|$ of the matrix whose columns are the elements of~$\bF$.
\end{theorem}
In Remark~\ref{remark:mofs2} below we will discuss the quantities $\g(\bF)$ in more detail. 
In the above case that $\bA = \bA_\graph$ stems from a graph $\graph$, Stanley~\cite{MR1116376} proved that the coefficient of $t^j$ in the Ehrhart polynomial of $\Zono(\bA_\graph)$ equals the number of  forests in $\graph$ with $j$ edges.

Each part of the above-mentioned triple comes with its own concepts, structures and natural generalizations, raising the questions if and how these translate along the triple. 
We now give an overview of the concepts, structures and generalizations which we investigate and state the corresponding results.

Note that for a graph $\graph$ the columns in $\bA_\graph$ form a subset of the so called root system of type A.
Hence our \emph{first} motivation for this paper is to describe the above story for subsets of root systems of type B/C/D and add in missing pieces. 


%
In the following, we will
phrase this setting in the language of root systems; see, e.g., \cite{humphreys_reflection_1990} for an introduction.
Root systems arise in the context of Lie theory
%
%
and are fully classified into five exceptional root systems and the following four infinite families of irreducible root systems:
\begin{itemize}
\item $A_{n-1} =\{\pm(\bb_i - \bb_j) \}_{i \neq j} $,
\item $B_n =  \{\pm(\bb_i - \bb_j) \}_{i \neq j} \cup \{ \pm (\bb_i + \bb_j) \}_{i \neq j} \cup \{ \pm \bb_i  \} $, 
\item $C_n =  \{\pm(\bb_i - \bb_j) \}_{i \neq j} \cup \{ \pm (\bb_i + \bb_j) \}_{i \neq j} \cup \{ \pm 2 \bb_i  \},$ 
\item $D_n =  \{\pm(\bb_i - \bb_j) \}_{i \neq j} \cup \{ \pm (\bb_i + \bb_j) \}_{i \neq j}$, 
\end{itemize}
where $\bb_1, \bb_2, \dots,\bb_n$ form an orthonormal basis of $\RR^n$. 
Root systems have associated hyperplane arrangements by taking the roots as the hyperplane normals.

%
The
generators $\bA_G = [ \be_j - \be_k : \, jk \in E ]$ of the acyclotope for an ordinary graph
$\graph = (V, E)$ form a subset of a root system of type A, and consequently, the acyclotope is
a Minkowski summand of the permutahedron; the graphic
arrangement, in parallel, arises from a subset of a root system of type~A.

Some combinatorial data of these hyperplane arrangements such as the number of regions, can be studied in terms of (signed) graphs \cite{zaslavskygeometryrootsystems,ZasGre}.
Signed graphs were introduced as a combinatorial model for subsets of root systems in Coxeter type B/C/D. 
They originated in the social sciences and have found applications also in biology, physics, computer science, and economics; see \cite{zaslavsky2012mathematical} for a comprehensive bibliography.

A \Def{signed graph} $\sgraph = (\ugraph,\sigma)$ consists of a graph $\ugraph = (V,E)$ and a
signature $\sigma$.
In the underlying graph $\ugraph$, the edge set $E$ may contain besides the usual, potentially multiple, links (two distinct endpoints) and loops (two endpoints that are the same), also halfedges (with only one endpoint) and loose edges (no endpoints), though the latter play no role in our work.
The signature $\sigma$ assigns each link and loop of $\ugraph$ either $+$ or $-$.
An ordinary graph can be realized by a signed graph all of whose edges are labeled with~$+$.

Signed graphs represent root systems as follows.\footnote{
This correspondence is one reason to leave out positive loops and loose edges when building
the incidence matrix; neither play a role in our work.
}  
As in the (unsigned) graph case, we can define an incidence matrix;
the precise definitions can be found in \Cref{sec:signedacyclotopes}; see also \Cref{fig:orientededges}.
An \Def{incidence matrix} of a given signed graph $\sgraph$ with $n$ nodes and $m$ edges is an
$n \times m$ matrix $\bA_\sgraph$ whose column corresponding to the edge $e$ equals
\begin{itemize}
\item $\be_j - \be_k$ or $\be_k - \be_j$ if $e$ is a positive link with endpoints $j$ and $k$,
\item $\be_j + \be_k$ or $-\be_k - \be_j$ if $e$ is a negative link with endpoints $j$ and $k$,
\item $\be_j$ if $e$ is a halfedge at $j$,
\item $2 \be_j$ or $-2 \be_j$ if $e$ is a negative loop at $j$.
\end{itemize}
The choices in the above list correspond to choosing a biorientation of $\sgraph$, in analogy
with $\bA_\graph$ depending on an orientation of $\graph$. In both cases, the combinatorial and
arithmetic data that we will compute are independent of the chosen (bi-)orientation.
Hence, a subset of roots of type A is captured by a graph consisting of links with all positive signature.
In order to encode subsets of type D we need to add links with negative signs, 
for type B we introduce halfedges and for type C we use negative loops.
Note that we will also consider signed graphs with halfedges and loops, which correspond to encoding a subset of roots of type B and C.

Parallel to the graphic case, we define the \Def{acyclotope} corresponding to a signed graph $\sgraph$ as
the zonotope $\Zono(\bA_\sgraph)$. 
Zaslavsky's original definition of the acyclotope in \cite[Section 4]{zaslavskyorientationsignedgraphs} is slightly different: 
 it is a centrally symmetric version, which is homothetic to our definition as it is a translation of the second dilate. 
  We chose the above definition because it features more natural arithmetic
properties and specializes to the case of unsigned graphs.

 Thus the acyclotope is a Minkowski summand of the respective (integral) Coxeter permutahedron
$\Pi^\ZZ(\Phi) := \sum_{\alpha \in \Phi^+} [0,\alpha]$
of the finite root system $\Phi$ with a choice $\Phi^+$ of positive roots, see,
e.g.,~\cite{ardila_arithmetic_2020}.

As already mentioned, the triple graphic matroids/graphic hyperplane arrangements/acyclotopes is phrased in terms of root systems of type A. In this paper we investigate the analogue triple for type B/C/D root systems and investigate the analogues for the mentioned structural results. 
The face structure of the acyclotope
in this setting goes back to the same Greene--Zaslavsky paper mentioned above~\cite{ZasGre}.
For the Ehrhart polynomial of acyclotopes we show the following result, extending results in \cite{ardila_arithmetic_2020}.
\begin{reptheorem}{thm:signedehrhart}
 The Ehrhart polynomial of the acyclotope $\Zono(\bA_\sgraph)$ equals
 \begin{equation}
  \ehr_{\Zono(\bA_\sgraph)}(t)
  = \sum_{F} 2^{\pc(F)+\lc(F)} \, t^{\nn-\tc(F)}\,,
 \end{equation}
 where the sum is over all $F \subseteq E$ such that $\subsgraph{F}$ is a pseudo-forest.
\end{reptheorem}




Our \emph{second} motivation is to investigate the analogues of the above results under duality. 
Generalizing the concept of duality for plane graphs, Whitney \cite{whitney} 
introduced \Def{matroidal duality}.

Let $\matroid = (E, \cI)$ be a matroid on the ground set $E$ with independent sets collected
in $\cI$. Its \Def{dual matroid} $\dualmatroid\coloneqq (E,\cI^{\triangle})$ is
defined via
\begin{align}
    \cI^{\triangle} \coloneqq \left\{ J \subseteq E \ | \ E \setminus J \ \text{is a
spanning set of } M \right\} ,
\end{align}
where a subset $S\subseteq E$ is called \Def{spanning} if it contains a basis.
For general background on matroids and terms that we will leave undefined, see,
e.g.,~\cite{welsh_matroid_2010}.
We are interested in the case that $M$ is \Def{representable} (over $\RR$), i.e., $S$ consists of the columns of a given matrix $\bA \in \RR^{n \times m}$ and independence refers to linear independence.
In this case, and under the (reasonable) assumption that $\bA$ has rank $n$, there is a well-known construction of a representation $\bA^\triangle$ for the dual matroid
 $M^\triangle$, see, e.g., \cite[Section~9.3]{welsh_matroid_2010}.
Namely, one uses elementary row operations on $\bA$ resulting in a matrix of the form $[\bW \, | \, \bI]
\in \RR^{n \times m}$ where we denote by $\bI$ the identity matrix of the appropriate dimension.
The matrix $[\bW \, | \, \bI]$ also represents $M$.
Now let $\bA^\triangle := [\bI \, | \, -\bW^{\top}] \in \RR^{ (m-n) \times m }$; by construction $\bA^\triangle$ represents $M^\triangle$.
Note that the rows in $\bA^\triangle$ form a basis for the kernel $\ker(\bA)$, so the described construction is a special case of what is known as \Def{Gale duality} or \Def{Gale transforms}.

Going back to the triple from the beginning, let us consider a matroid $\matroid$ representable over $\QQ$ with representation $\bA\in\ZZ^{\nn\times\ned}$, the corresponding hyperplane arrangement $ \cH(\bA)$ and the lattice zonotope $\Zono(\bA)$.
For all three objects we can construct their duals: the dual matroid $ \matroid^\triangle$ with representation $\bA^\triangle$, the corresponding hyperplane arrangement $\cH(\bA^\triangle)$ and zonotope $\Zono(\bA^\triangle)$.
McMullen~\cite{mcmullernzonotopes} studied the face structure of  $\Zono(\bA^\triangle)$ and showed that it is uniquely described in terms of combinatorial data from the matroid $\matroid$.
However, in general the dual representation $\bA^\triangle$ may not be integral, its construction involves a
number of choices, and therefore the Ehrhart polynomial of the dual zonotope  $\Zono(\bA^\triangle)$ is a priori not uniquely defined.

We are particularly interested in the case when $\bA = \bA_\sgraph$ or $\bA = \bA_\graph$ is the incidence matrix of a (signed) graph, i.e., consists of roots of type A/B/C/D. 
%
We construct and investigate the properties of the corresponding triple for root systems of type A/B/C/D and introduce the \Def{tocyclotope} as the (signed-)graphic dual construction of the acyclotope.
The face structure of the tocyclotope in the case of an ordinary graph can once more be found in~\cite{ZasGre} in terms of the hyperplane arrangement and the name comes from the fact that their vertices corresponds to the totally cyclic
orientations\footnote{
An orientation is called \Def{totally cyclic} if every edge belongs to a coherently
oriented cycle.
}. 

%
%

In \Cref{sec:graphtocyclotopes}, we describe a combinatorial construction of the tocyclotope based on the fact that the incidence matrix $\bA_\graph$ of a graph $\graph$ is totally unimodular and compute its Ehrhart polynomial:
\begin{reptheorem}{thm:cographicehrhart}
Let $\graph=(V,E)$ be a simple and connected
graph. Then the Ehrhart polynomial of the tocyclotope $\Zono(\bA_G^\triangle)$ is 
\begin{align}
    \ehr_{\Zono(\bA_G^\triangle)}(t)
    = \sum_{k=0}^{m-n+1} d_k \, t^{k}
\end{align}
where the coefficient $d_k$ equals the number of complements of spanning sets in~$\graph$ of size~$k$, which equals the number of forests of size $m-k$ in $G^{\Delta}$.  
\end{reptheorem}
Furthermore, we complete the connections along the dual triple by commenting on how tocyclotopes connect to
flows on $\graph$ and cographic arrangements. 
%
The translation of combinatorial data to the matroidal dual recently attracted attention in the area of scheduling periodic time tables~\cite{liebchen_modeling_2007}  and tropical geometry for the type-A case~\cite{bortoletto_tropical_2024}.
Here, integer points in  the tocyclotope of a graph correspond to cycle offsets of a periodic event scheduling problem.

In  \Cref{sec:Steinitz} we push the triple even further to polytopes and consider the special case where
our graph is Steinitz, i.e., $3$-connected, simple and planar and hence the $1$-skeleton of a $3$-dimensional polytope. For this class of graphs polytopal duality, matroid duality and graph duality coincide;
in particular, acyclic orientations are dual to totally cyclic orientations on the underlying graph structure. 
We construct the normal vectors of a cographic hyperplane arrangement as the column vectors of a boundary map stemming from the cellular chain complex associated to the CW-complex given by the face structure of~$P$.

However, in particular in the signed graph case, we stress again that it is a priori not clear how to ensure that the tocyclotope is a \emph{lattice} zonotope and that the Ehrhart polynomial is unique.
This is because the incidence matrix $\bA_\sgraph$ of signed graphs is not totally unimodular.
Therefore, we first turn to the more general setting of arbitrary lattice zonotopes in \Cref{sec:latticegaledual}.
The construction and arithmetic of the tocyclotope suggest a general duality
construction for zonotopes, one that was already employed by
McMullen~\cite{mcmullernzonotopes} and D'Adderio--Moci~\cite{dadderiomoci,dadderio_arithmetic_2013}:
starting with $\bA \in \RR^{n \times m}$, construct a matrix that represents the matroid dual to that of $\bA$; 
McMullen described the face structure of its associated zonotope entirely
from the data of $\bA$, and D'Adderio--Moci developed the general notion of an
\emph{arithmetic matroid}, whose duality concepts apply here. 
%
We define the \Def{lattice Gale dual} of $\bA$ as the matrix $\bD\in\ZZ^{\ned
\times (\ned-\nn)}$, which consists of a $\ned-{\nn}$ \emph{lattice} basis vectors of $\ker(\bA)\cap \ZZ^\ned$ as column vectors. 
We give a concrete computation that describes the Ehrhart polynomial of the zonotope of the dual arithmetic matroid, the \Def{lattice Gale zonotope} $\Zono(\bD^{\top})\subset\RR^{\ned-\nn}$, in terms of $\bA$:
\begin{reptheorem}{thm:galezonehrhart}
Let $\bA\in \ZZ^{{\nn}\times \ned}$ be of rank $n$, with lattice Gale dual $\bD\in\ZZ^{\ned \times (\ned-\nn)}$.
Then we can compute the Ehrhart polynomial of the associated lattice Gale zonotope as
\begin{align}
 \ehr_{\Zono(\bD^{\top})}(t)=\sum_{S}\frac{\g(\bA_S)}{\g(\bA)}\ t^{m-\lvert S\rvert}
\end{align}
where the sum is over all spanning sets $S\subseteq[m]$ in the matroid represented by~$\bA$ and $\g(\bA)$ is defined as in \Cref{thm:stanley}.
\end{reptheorem}

Finally, in \Cref{sec:signedtoyclotopes}, we give an analogous tocyclotope construction for signed graphs, 
establish a connection to the Gale lattice dual and use this in order to guarantee uniqueness of the tocyclotope for signed graphs up to unimodular equivalence. Hence, we can compute the well-defined Ehrhart polynomial of the tocyclotope for signed graphs and obtain the following result which fits exactly into the duality framework.
\begin{reptheorem}{thm:signedtocyclotopeehr} 
Let $\sgraph$ be a connected signed graph whose incidence matrix has full rank. 
Choose a connected basis $T\subseteq E$ that contains a halfedge if $\sgraph$ contains a halfedge.
Then the Ehrhart polynomial of the tocyclotope $\Zono([\bI \, |-\bB^{\top}])\subseteq\RR^{\ned-\nn}$ is
\begin{equation}
 \ehr_{\Zono([\bI \, |-\bB^{\top}])}(t)= \begin{cases}
                  \sum_{S} 2^{\mplc(S)} t^{\ned-\lvert S\rvert} & \text{if $\sgraph$ contains a halfedge,}\\
                  \sum_{S} 2^{\mplc(S)-1} t^{\ned-\lvert S\rvert} & \text{if $\sgraph$ does not contain a halfedge,}
                 \end{cases}
\end{equation}
where the sums run over all sets $S\subseteq E$ that contain a basis of $\sgraph$, 
and $\mplc(S)$ is some combinatorial data of the signed graph $\sgraph$ defined in detail below.
\end{reptheorem}

\section{Acyclotopes for signed graphs}\label{sec:signedacyclotopes}

In this section we will first recall some theory of signed graphs and then extend results in \cite{ardila_arithmetic_2020} to give graph-theoretic interpretations for the Ehrhart polynomial of the signed acyclotope. 
 
\subsection{Preliminaries: Signed graphs and their geometry}\label{ssec:signedgraphs}
We now elaborate on the construction of an incidence matrix of a signed graph, hinted at in
the introduction. To begin, we once more stress their relation to root vectors of type B/C/D;
excellent background references are \cite{zaslavskygeometryrootsystems,zaslavsky_matrices_2010}.

An \Def{orientation} of a signed graph $\sgraph = (\ugraph,\sigma)$ is an assignment $\tau$ from the set of node-edge incidences to $\{\pm\}$ such that $\sigma(e)=-\tau(e,v)\tau(e,u)$ for very edge $e=\{u,v\}$.
Equivalently, choosing a \Def{bidirection} $\tau: E \times V \to \{ \pm \}$ for an unsigned graph $\ugraph=(V,E)$ first and setting $\sigma(e)=-\tau(e,v)\tau(e,u)$ defines an oriented signed graph $\sgraph=(\ugraph, \sigma)$ with orientation $\tau$.
Hence, oriented signed graphs and bidirected graphs are equivalent objects.
We can interpret this as follows: 
if $\tau(e,v)=+$ the edge $e$ enters node $v$, i.e., the head of the node-edge incidence $(e,v)$ points towards $v$,
if $\tau(e,v)=-$ the edge $e$ exits node $v$, i.e., the head of the node-edge incidence
$(e,v)$ points away from $v$.\footnote{
Thus positive edges get oriented consistently with orienting an unsigned graph, whereas
negative edges get oriented by choosing one of the charming adjectives \emph{introverted} or
\emph{extroverted}.
}
See \Cref{fig:orientededges}.

\begin{figure}[b]
\begin{subfigure}[t]{.55\textwidth}
  \centering
  \includegraphics[]{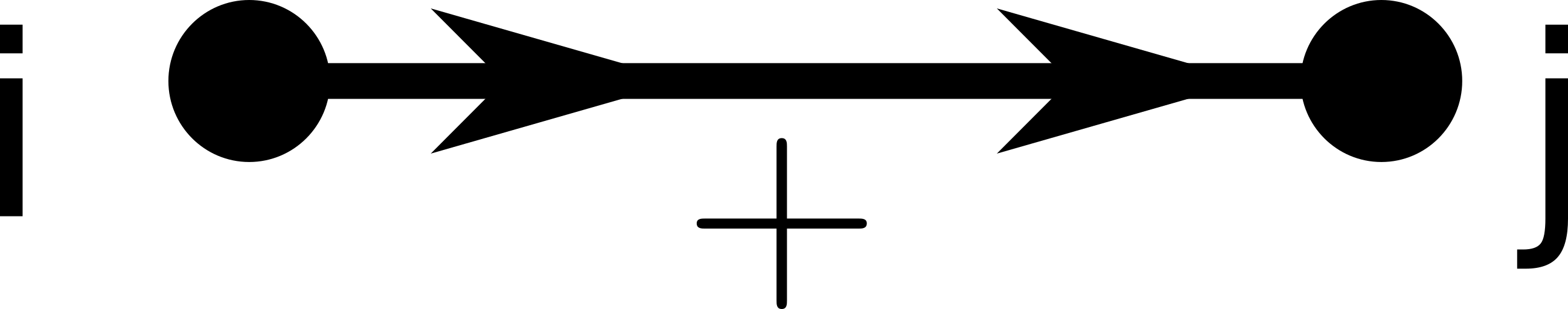}\hspace{5ex}
  \includegraphics[]{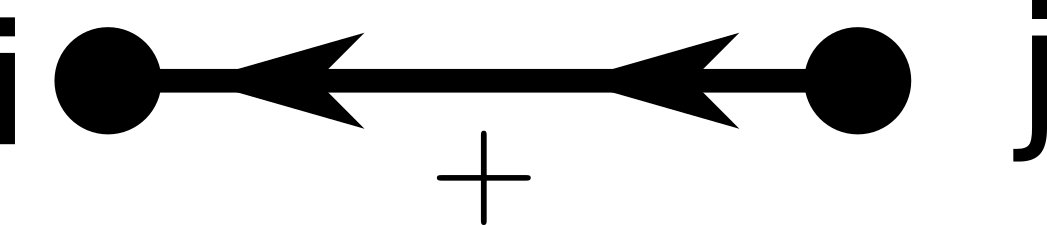}
\caption{positive link (bi)oriented from $i$ to $j$ (left) and from $j$ to $i$ (right), i.e., $\tau(\{i,j\}, i)=-$ and $\tau(\{i,j\}, j)=+$ (left) and with flipped signs on the right} 
\end{subfigure}\hfill%
\begin{subfigure}[t]{.423\textwidth}
  \centering
  \includegraphics[]{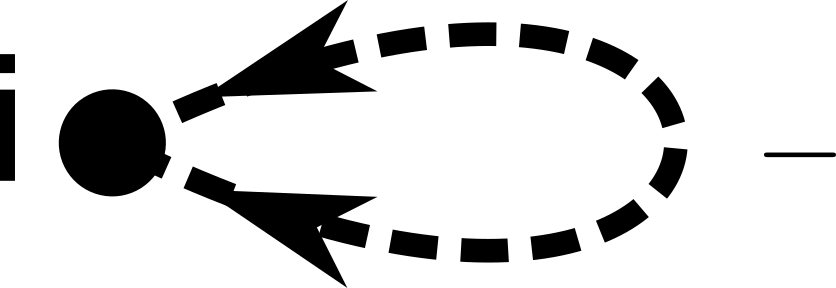}\hspace{4ex}
  \includegraphics[]{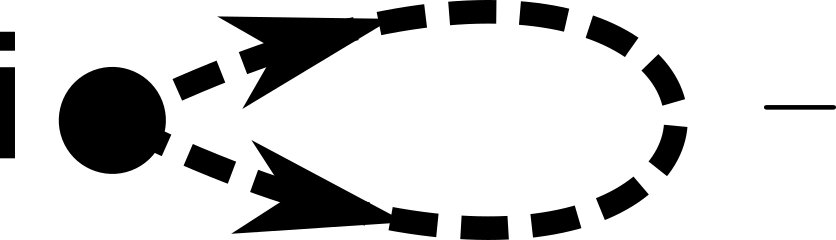}
\caption{negative loop (bi)oriented extroverted (left) and introverted (right), i.e., $\tau(\{i,i\}, i)=+$ (left) and $\tau(\{i,i\}, i)=-$ (right)}
\end{subfigure}

\vspace{2ex}

\begin{subfigure}[t]{.55\textwidth}
  \centering
  \includegraphics[]{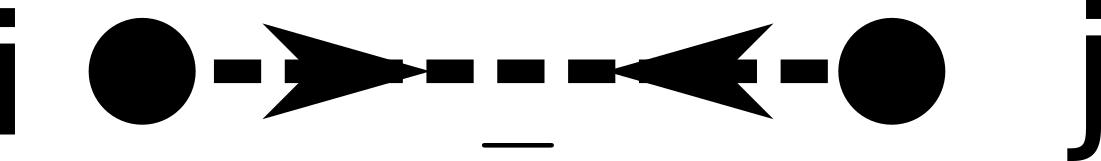}\hspace{5ex}
  \includegraphics[]{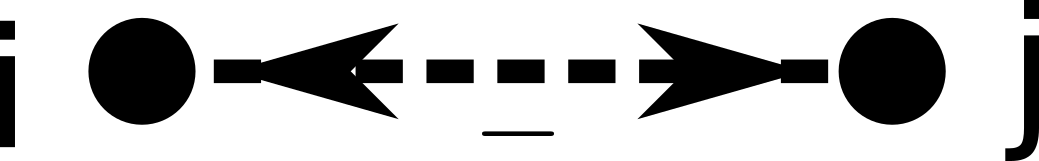}
\caption{negative link (bi)oriented introverted (left) and extroverted (right), i.e.,  $\tau(\{i,j\}, i)=-=\tau(\{i,j\}, j)$ (left) and with flipped signs on the right} %
\end{subfigure}\hfill%
\begin{subfigure}[t]{.42\textwidth}
  \centering
  \includegraphics[]{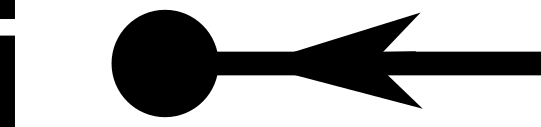}\hspace{4ex}
  \includegraphics[]{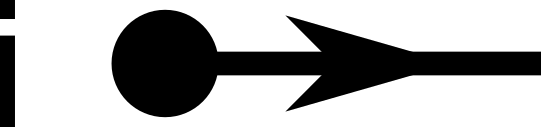}
\caption{halfedge (no signature) pointing towards~$i$, i.e., $\tau(\{i\},i)=+$ (left) and  pointing away from $i$, i.e., $\tau(\{i\},i)=-$ (right)} %
\end{subfigure}
 \caption{All different types of bioriented edges.} \label{fig:orientededges}
\end{figure}

%
%

For an oriented signed graph $\sgraph$ without positive loops or loose edges, we define the \Def{incidence matrix}
$\bA_\sgraph\in\RR^{\nn\times\ned}$ by 
\begin{align}
 (\bA_\sgraph)_{v,e}=\begin{cases}
          0 & \text{ if $v$ and $e$ are not incident,}\\
          +1 & \text{ if $e$ enters $v$, i.e., }\tau(v,e)=+\,,\\
          -1 & \text{ if $e$ exits $v$, i.e., }\tau(v,e)=-\,,\\
          \pm 2 & \text{ if $e$ is a negative loop at $v$ and }\tau(v,e)=\pm \text{, respectively.}
         \end{cases}
\end{align}



For a subset $R\subseteq E$ of edges of a signed graph $\sgraph=(\ugraph, \sigma)$ with $\ugraph=(V,E)$ we define the \Def{subgraph} $\subsgraph{R}$ to be the signed graph with the underlying graph $\subugraph{R}=(V,R)$ and the signature $\sigma$ restricted to $R$. 
With this definition of subgraph, the incidence matrix of $\bA_{\subsgraph{R}}\eqqcolon\bR$ is precisely the matrix formed by columns of $\bA_\sgraph$ indexed by $R$.
We recall some further notions from the theory of signed graphs:
\begin{itemize}
   \item A \Def{path} is a sequence $(v_1,e_1,v_2,e_2,\dots,e_n,v_{n+1})$ of nodes $v_i$ and edges $e_i$, such that $e_i=\{v_i,v_{i+1}\}$ and all the edges and vertices in the sequences are pairwise distinct. 
   \item A signed graph $\sgraph$ is \Def{connected} if there exists a path between any two nodes.
   \item A \Def{circle} is a closed path, i.e., a sequence $(v_1,e_1,v_2,e_2,\dots,e_n,v_{1})$ where only the first and last node are equal. A loop is a circle but we will usually treat loops separately from circles. 
%
    \item A \Def{signed tree} is a connected signed graph with no circles, loops, or halfedges.
          See \Cref{subfig:signed_tree}.
    \item A \Def{(signed) halfedge-tree} is a connected (signed) graph with no circles or loops, and a single halfedge.
    See \Cref{subfig:halfedge-tree}.
    \item  A \Def{(signed) loop-tree} is a connected (signed) graph with no circles or halfedges, and a single negative loop.
    See \Cref{subfig:loop-tree}.
    \item A \Def{(signed) pseudo-tree} is a connected (signed) graph with no loops or halfedges that contains a single circle with an odd number of negative edges.
    See \Cref{subfig:pseudo-tree}.
    \item A \Def{signed pseudo-forest} is a signed graph whose connected components are signed trees, signed halfedge-trees, signed loop-trees, or signed pseudo-trees.
    See \Cref{subfig:pseudo-forest}.
    \item A \Def{circuit} is a subgraph with an inclusion minimal set of edges that is not a pseudo-forest. 
    For signed graphs this can be a circle with an even number of negative edges,
and a \Def{handcuff}, i.e., a path (possibly consisting of only one node) that on each of its two (possibly identical) end-nodes is connected to a negative circle, halfedge, or negative loop. See \Cref{subfig:balanced_circle,subfig:handcuff,subfig:handcuff_degenerate}.
    \item A \Def{source (sink)} is a node $s$ with only outward (inward) pointing node-edge incidences, i.e., $\tau(s,e)=-$ ($\tau(s,e)=+$) for all edges $e$ incident to the source (sink) $s$.
    \item A \Def{(bidirected) cycle} is an oriented circuit that has neither sinks nor sources.
    \item An oriented signed graph is called \Def{acyclic} if it does not contain any cycles.
     It is called \Def{totally cyclic} if every edge is contained in a cycle.
\end{itemize}
\begin{figure}
 \begin{subfigure}[t]{.33\textwidth}
 \centering
  \includegraphics[width=.9\linewidth]{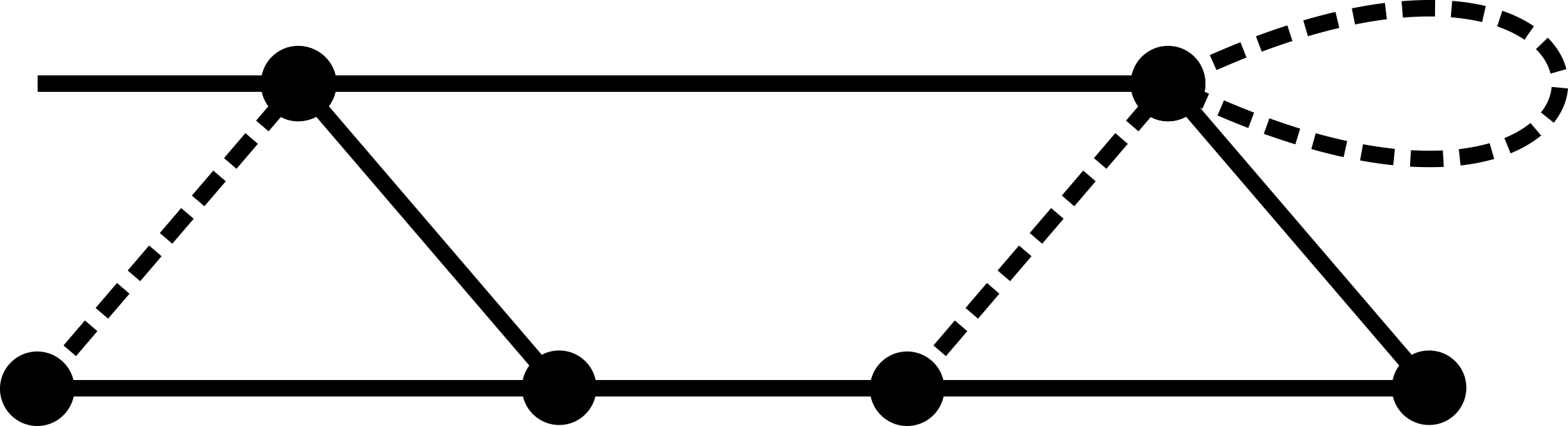}
  \caption{full graph }\label{subfig:signed_graph}
 \end{subfigure}\hfill%
 \begin{subfigure}[t]{.33\textwidth}
 \centering
  \includegraphics[width=.9\linewidth]{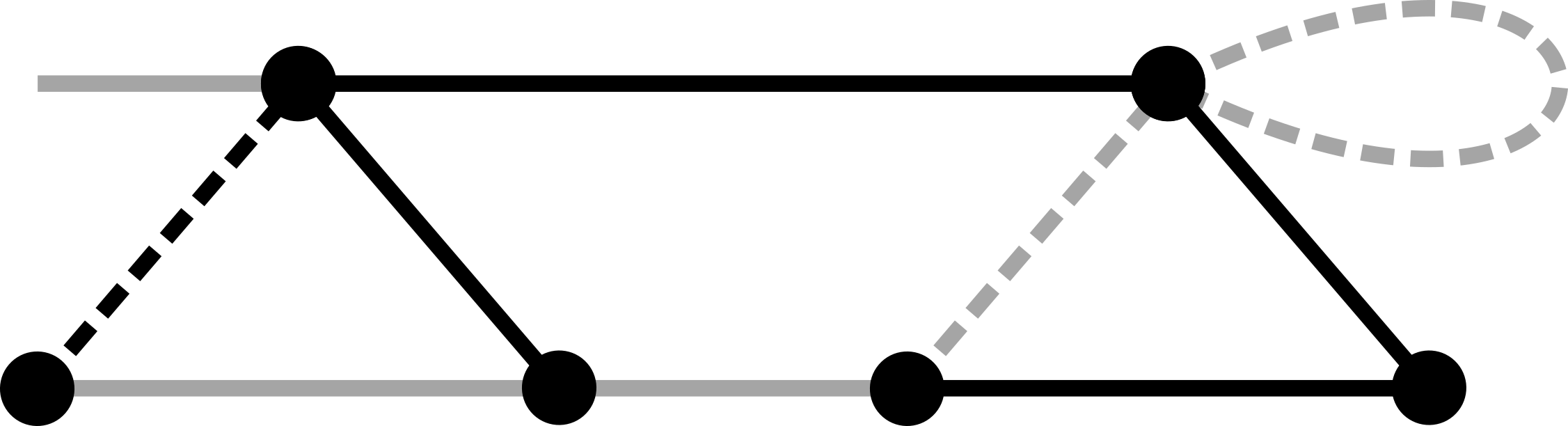}
  \caption{signed tree }\label{subfig:signed_tree}
 \end{subfigure}\hfill%
 \begin{subfigure}[t]{.33\textwidth}
 \centering
  \includegraphics[width=.9\linewidth]{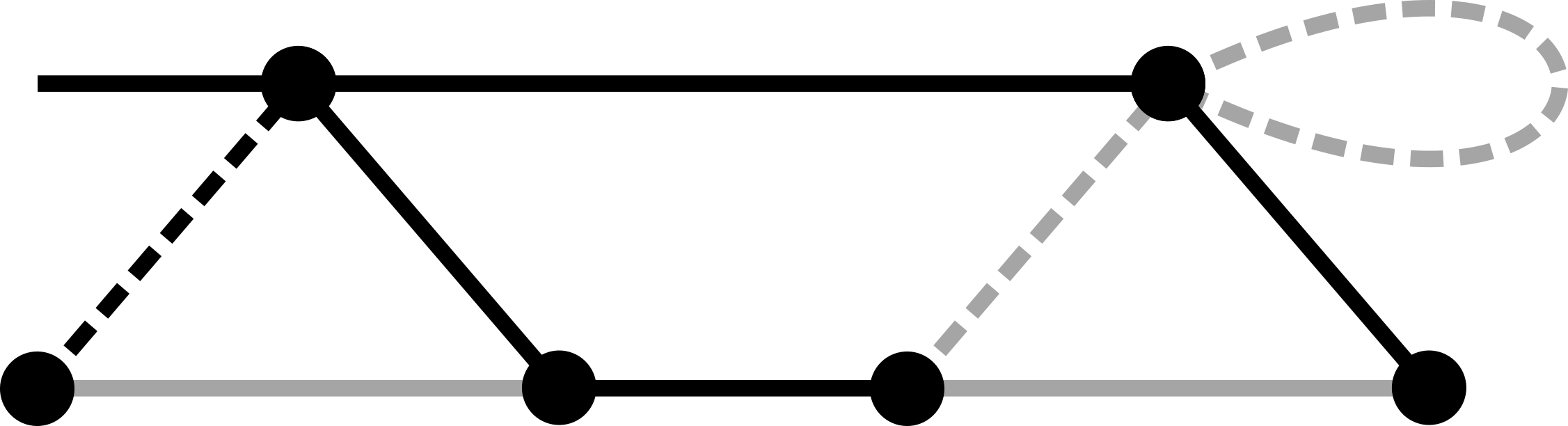}
  \caption{(signed) halfedge-tree}\label{subfig:halfedge-tree}
 \end{subfigure}

 \begin{subfigure}[t]{.33\textwidth}
 \centering
  \includegraphics[width=.9\linewidth]{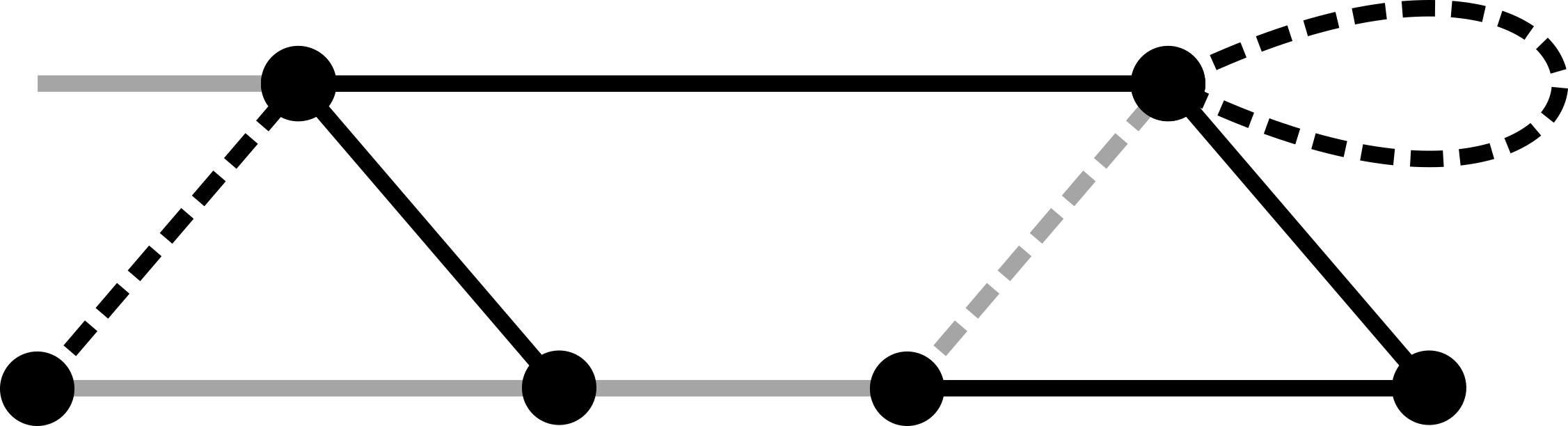}
  \caption{(signed) loop-tree}\label{subfig:loop-tree}
 \end{subfigure}\hfill%
  \begin{subfigure}[t]{.33\textwidth}
 \centering
  \includegraphics[width=.9\linewidth]{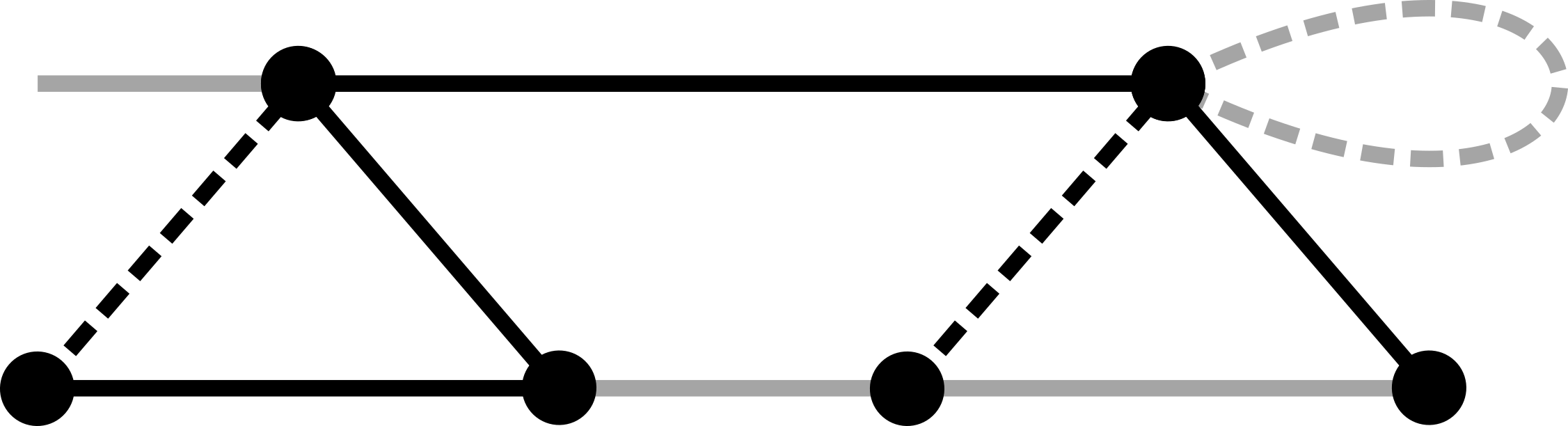}
  \caption{(signed) pseudo-tree}\label{subfig:pseudo-tree}
 \end{subfigure}\hfill%
  \begin{subfigure}[t]{.33\textwidth}
 \centering
  \includegraphics[width=.9\linewidth]{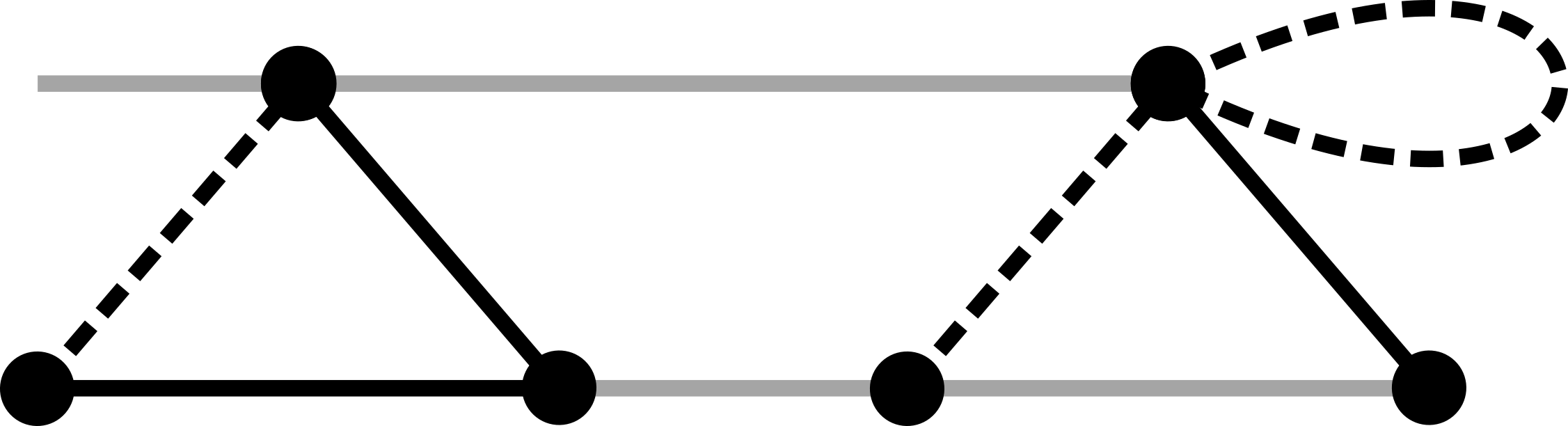}
  \caption{inclusion-maximal pseudo-forest }\label{subfig:pseudo-forest}
 \end{subfigure}
 
 \begin{subfigure}[t]{.33\textwidth}
 \centering
  \includegraphics[width=.9\linewidth]{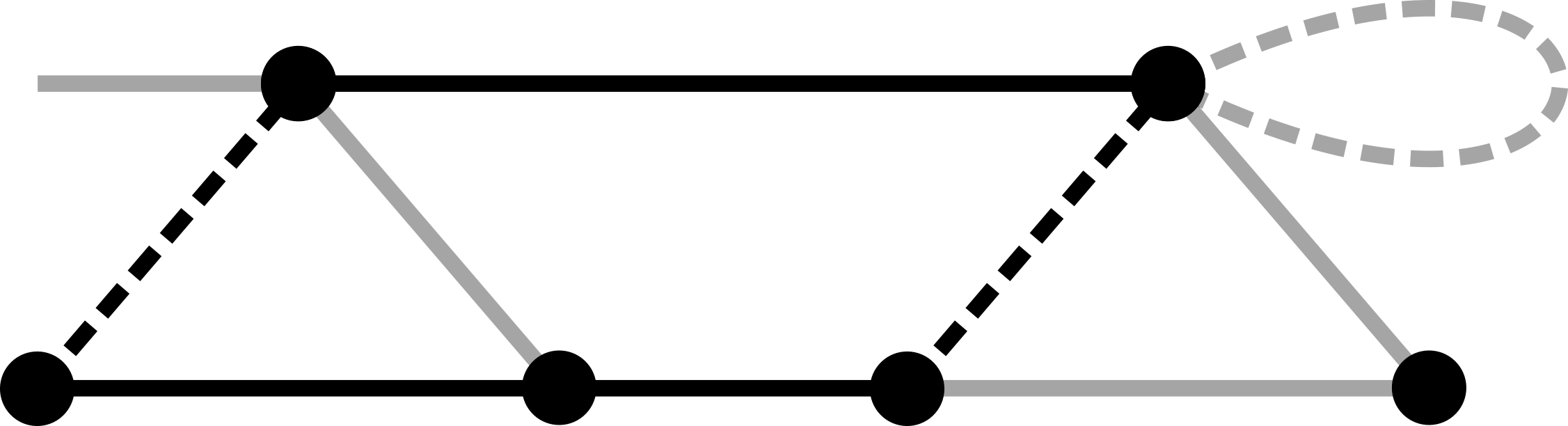}
  \caption{balanced circle}\label{subfig:balanced_circle}
 \end{subfigure}\hfill%
 \begin{subfigure}[t]{.33\textwidth}
 \centering
  \includegraphics[width=.9\linewidth]{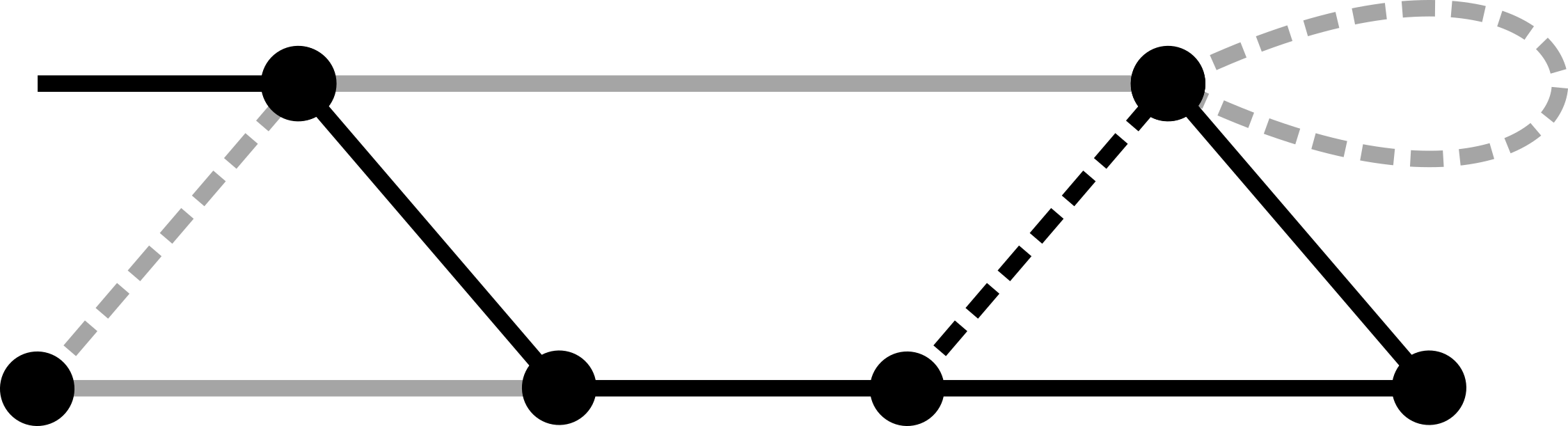}
  \caption{handcuff with  connecting path}\label{subfig:handcuff}
 \end{subfigure}\hfill%
 \begin{subfigure}[t]{.33\textwidth}
 \centering
  \includegraphics[width=.9\linewidth]{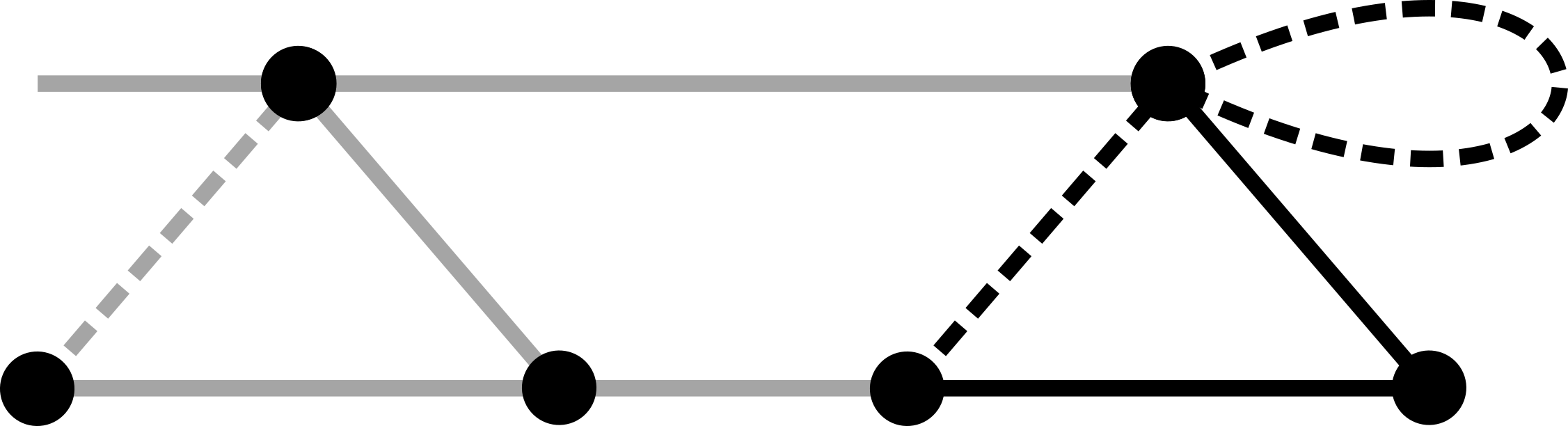}
  \caption{handcuff with degenerate connecting path}\label{subfig:handcuff_degenerate}
 \end{subfigure}
\caption{Various types of subgraphs.}\label{fig:subgraphs}
\end{figure}
%
We want to emphasize that we carefully distinguish between the notions of circles, circuits, and cycles, which tend to have non-uniform  meanings throughout the literature.

%

For a signed pseudo-forest $\sgraph$, let $\tc(\sgraph)$, $\lc(\sgraph)$ and $\pc(\sgraph)$ be the number of tree components, loop-tree components, and pseudo-tree components, respectively.
For a subset $F\subseteq E$ of edges we use the shorthand $\tc(F)\coloneqq \tc(\subsgraph{F})$ to count the number of signed tree components in the edge subgraph $\subsgraph{F}$.
We define $\pc(F)$ and $\lc(F)$ analogously.
Note that the edge subgraph $\subsgraph{F}$ still has the same number $\nn$ of nodes as the signed graph $\sgraph$;
in particular $\tc(\emptyset)= \tc(\subsgraph{\emptyset})=\nn$.
For example, the subgraph in \Cref{subfig:pseudo-forest} has one loop-tree component, one pseudo-tree component, and no tree component.


A signed graph $\sgraph = (\ugraph, \sigma)$ is called \Def{balanced} if it does not contain any halfedges and every circle has an even number of negative edges. 
An unsigned graph can be realized by a signed graph all of whose edges are labelled with $+$; it is automatically balanced.
Vice versa, a balanced graph can be converted into a signed graph with only positive edges by switching operations, that is, for a fixed vertex $v$ flipping the sign of $\tau(v,e)$ for all node-edge incidences involving $v$. 
In terms of the incidence matrix, switching means flipping all signs in one row. 
The incidence matrix of a connected signed graph is full rank if and only if the graph is not balanced.
For more on switching equivalences and balance of signed graphs see, e.g.,  \cite{zaslavskysignedgraphs,zaslavsky_matrices_2010}.


To every signed graph $\sgraph$ we can associate a \Def{signed graphic matroid}
$\matroid(\sgraph)=(E,\mathcal{I})$, also called the \Def{bias matroid} of
$\sgraph$, which is the representable matroid with ground set $E$, indexing the columns of the  incidence matrix $A $ and independent sets formed by selections of linearly independent columns in $\bA_\sgraph $.
Switching operations preserve all combinatorial data in the signed graphic matroid.

\begin{prop}[Zaslavsky~\cite{zaslavskysignedgraphs}]\label{thm:signedindependence}
We recall the signed graphic meaning of the relevant matroid notions:
\begin{enumerate}[{\rm (1)}]
 \item A subset $F\subseteq E$ is an independent set in $\matroid(\sgraph)$ if and only if the subgraph $\subsgraph{F}$ is a signed pseudo-forest. 
 \item  A subset $T\subseteq E$ is a basis in $\matroid(\sgraph)$ if and only if the subgraph $\subsgraph{T}$ is an inclusion-maximal signed pseudo-forest in $\sgraph$.  
 \item A subset $C\subseteq E$ is a circuit in $\matroid(\sgraph)$ if and only if the subgraph $\subsgraph{C}$ is a circuit together with isolated vertices (the set of isolated vertices might be empty).
  \item A subset $C\subseteq E$ is dependent in $\matroid(\sgraph)$ if and only if the subgraph $\subsgraph{C}$ contains a circuit.
\end{enumerate}
\end{prop}

 Note that inclusion-maximal edge sets that form a pseudo-forest as subgraphs are not necessarily connected, hence not necessarily a (signed) tree, loop-tree, halfedge-tree, or pseudo-tree; see, e.g., \Cref{subfig:pseudo-forest}.
 We will call an edge set $S\subseteq E$ \Def{spanning} if it contains an inclusion-maximal pseudo-forest. Again this does not imply that the subgraph  is connected.



A \Def{coloop} is an edge in a signed graph that corresponds to a coloop in the signed graphic matroid, i.e., an element that is contained in every basis. 
A coloop is an edge whose deletion makes an
unbalanced component balanced, or a bridge connecting two components of which
at least one is balanced \cite[Theorem 5.1]{zaslavskysignedgraphs}.

\subsection{The Ehrhart polynomial of the acyclotope of a signed graph}

 Our definition of $\Zono(\bA_\sgraph)$ implicitly depends on a choice of
orientation for each edge of $\sgraph$ (except loops and halfedges). 
 However, any property of the acyclotope discussed here is independent of these choices.
 For example, the face structure of the zonotope is determined by the combinatorial structure of the corresponding hyperplane arrangement, i.e., the poset of intersections. This in turn does not depend on the choice of orientation of the hyperplane normals.

The following result generalizes the Ehrhart polynomial of the type-B root
polynomial~\cite{ardila_arithmetic_2020}.
Recall that $\tc(\sgraph)$, $\lc(\sgraph)$ and $\pc(\sgraph)$ denote the number of tree components, loop-tree
components, and pseudo-tree components of $\sgraph$, respectively.
\begin{theorem}\label{thm:signedehrhart}
 The Ehrhart polynomial of the acyclotope $\Zono(\bA_\sgraph)$ equals
 \begin{equation}
  \ehr_{\Zono(\bA_\sgraph)}(t)
  = \sum_{F} 2^{\pc(F)+\lc(F)} \, t^{\nn-\tc(F)}\,,
 \end{equation}
 where the sum is over all $F \subseteq E$ such that $\subsgraph{F}$ is a pseudo-forest.
\end{theorem}
We recall that the number of nodes in every edge subgraph $\subsgraph{F}$ is the same as
the number $n$ of nodes in $\sgraph$. 
Similarly, the number $\tc(F)$ of signed tree components
 counts all the isolated vertices; e.g., for a graph with $\nn$ nodes, $\tc(\emptyset)=\nn$.

The following lemma has appeared in various guises; see \cite[Lemmas~4.9 \& 4.10]{ardilacastillohenley}, \cite[Proposition~4.2]{kotnyek_generalization_2002}, and \cite[Lemma~8A.2]{zaslavskysignedgraphs}.

\begin{lem} 
\label{lem:relvolpi}
 Let $\bF$ be a linearly independent subset of the columns of $\bA_\sgraph$.
 The corresponding subset $F$ of edges of $\sgraph$ forms a pseudo-forest as subgraph.
 Then 
 \begin{align}
\g(\bF)= 2^{\pc(F)+\lc(F)} \quad\text{and}\quad \lvert F\rvert=\nn-\tc(F)\,,
 \end{align}
 where $\g(\bF)$ is as in \Cref{thm:stanley}.
\end{lem}

\begin{proof}[Proof of \Cref{thm:signedehrhart}]
We apply Stanley's \Cref{thm:stanley}.
Linearly independent subsets of $\bA_\sgraph$ correspond, by \Cref{lem:relvolpi}, to
pseudo-forests of $\sgraph$. \Cref{lem:relvolpi} also gives the dimension ($n-\tc(F)$) and
volume ($2^{\pc(F)+\lc(F)}$) of the parallelepiped associated with a given linearly
independent set. As we will discuss in \Cref{remark:mofs2}, the latter volume equals~$\g(\bF)$.
\end{proof}

\section{Tocyclotopes for oriented graphs and the Flow Space}\label{sec:graphtocyclotopes}

In this section we investigate the duality translations for type A root systems, i.e., for simple oriented graphs. In \Cref{sec:signedtoyclotopes}, we extend the results of this section to root systems of type B/C/D.

We start with an (unsigned, simple) graph $\graph=(V,E)$ with incidence matrix $\bA_\graph
\in \RR^{ n \times m } $; it comes
with a natural block form given by the connected components of $\graph$, and so we 
will assume that $\graph$ is connected.
The matroid $M$ defined by $\bA_\graph$ can be given in terms of $\graph$ (i.e., $M$ is a
\Def{graphic matroid}): such a reduction corresponds to choosing a basis of the corresponding matroid, i.e., a spanning tree, whose edges then correspond each to the the columns of the identity matrix. 

We first modify $\bA_\graph$ (which has rank $n-1$) to a full-rank matrix that still represents $M$, using a
well-known construction. Namely, one uses elementary row operations to create a row of zeros and then
discards the latter. The result is a matrix of the form $[\bW \, | \, \bI] \in \RR^{ (n-1)
\times m }$ known as a \Def{network matrix} of $\graph$. It can be constructed, e.g., via a
spanning tree of $\graph$, whose edges correspond to the identity matrix.

Matroid duality yields a representation matrix $\bA_\graph^{\triangle} = [\bI \, | \, -\bW^{\top}]
\in \RR^{ (m-n+1) \times m }$ for $M^\triangle$, which is called the \Def{cographic matroid}. It
can be described purely in terms of $\graph$: its ground set is again $E$ and the 
independent sets are precisely the complements of spanning sets of $\graph$.
A \Def{spanning set} of a connected graph is a subset of edges whose  subgraph
contains a spanning tree.

The \Def{tocyclotope} of $\graph$ is the zonotope $\Zono(\bA_G^\triangle)$.
We note that there were choices involved when constructing $\bA_G^\triangle$; however,
different choices correspond to resulting matrices that are unimodularly equivalent (since
elementary row operations are unimodular), Thus, $\Zono(\bA_G^\triangle)$ is unique up to unimodular equivalence.
We make use of Stanley's \Cref{thm:stanley} about Ehrhart polynomials of lattice zonotopes.
%
This yields the companion result to Stanley's~\cite{MR1116376} Ehrhart polynomial structure of $\Zono(\bA_\graph)$ mentioned in the Introduction.

\begin{theorem}\label{thm:cographicehrhart}
Let $\graph=(V,E)$ be a simple and connected
graph. Then the Ehrhart polynomial of the tocyclotope $\Zono(\bA_G^\triangle)$ is 
\begin{align}
    \ehr_{\Zono(\bA_G^\triangle)}(t)
    = \sum_{k=0}^{m-n+1} d_k \, t^{k}
\end{align}
where the coefficient $d_k$ equals the number of complements of spanning sets $S^C$ in~$\graph$ of size~$k=\lvert S^C\rvert$.
This also equals the number of forests of size $k$ in $G^{\Delta}$ if the graph $G$ was planar.
\end{theorem}

\begin{proof} 
Both $\bA_\graph$ and $\bA_\graph^{\triangle}$ are totally unimodular. 
Further, the column vectors of $\bA_\graph^{\triangle}$ are, by definition, linearly independent if
and only if they induce complements of spanning sets on $\graph$, as these induce the independent sets of the cographic matroid.
Now apply \Cref{thm:stanley}.
\end{proof}

\begin{example}
The tocyclotope of the complete graph $K_4$ is the 3-permutahedron. We can thus calculate
its Ehrhart polynomial as follows. The linear coefficient $d_1$ equals $6$ since every edge of
$K_4$ is a complement of a spanning set. Every choice of two edges of $K_4$ is a complement
of a spanning set and so the second coefficient $d_2$ equals $15$. Every choice of three edges besides the ones incident
to a single node is a complement of a spanning set and thus $d_3 = 16$.
In total we obtain
\begin{align}
    \ehr_{\Zono(\bA_{K_4}^\triangle)}(t) = 16 \, t^3 + 15 \, t^2 + 6 \, t + 1 \, .
\end{align}
\end{example}

We briefly comment on how $\Zono(\bA_\graph^\triangle)$ connects to flows on $\graph$ and its cographic arrangement;
here we assume that $\graph$ does not contain isthmi (in the graph case isthmi are bridges).
The \Def{flow space}
(also called the \emph{cycle space})
of~$\graph$ is defined as the kernel of $\bA_\graph$, which is an $(m-n+1)$-dimensional subspace of
$\RR^m$.
The \Def{cographic arrangement} $\cH(\bA_\graph^\triangle)$ is the arrangement induced by the coordinate
hyperplanes in $\RR^m$ on $\ker(\bA_\graph)$.
Greene and Zaslavsky~\cite{ZasGre} showed that the regions of
the cographic arrangement are in one-to-one correspondence with the totally cyclic orientations of~$\graph$. 

\begin{lem}\label{lem:flow_space}
Let $\graph=(V,E)$ be a simple, connected, and bridgeless graph.
The linear surjection $\bA_\graph^{\triangle}: \RR^m \to \RR^{ m-n+1 }$ maps the flow space
$\ker(\bA_\graph)$ bijectively to $\RR^{ m-n+1 }$. Thus the columns of the matrix $\bA_\graph^{\triangle}$ are normal vectors for an isomorphic copy of the cographic arrangement living
in~$\RR^{ m-n+1 }$.
\end{lem}

\begin{proof}
We need to prove $\ker(\bA_\graph) \cap \ker(\bA_\graph^\triangle) = \{ \bzero \}$. 
By construction, each row of $\bA_\graph$ is
perpendicular to each row of $\bA_\graph^\triangle$. Thus, any $\bw \in \ker(\bA_\graph) \cap
\ker(\bA_\graph^\triangle)$ is perpendicular to all row vectors in both $\bA_\graph$ and in 
$\bA_\graph^\triangle$, which implies $\bw = \bzero$.
\end{proof}

Thus, we may think of $\bA_\graph^\triangle$ as simultaneously generating the cographic
arrangement $\cH(\bA_\graph^\triangle)$ and the tocyclotope $\Zono(\bA_G^\triangle)$, giving
rise to geometric (e.g., the vertices of $\Zono(\bA_G^\triangle)$ are given by totally
cyclic orientations of~$\graph$) and arithmetic (e.g., \Cref{thm:cographicehrhart}) structures.


\section{Steinitz Graphs and Triple Duality}\label{sec:Steinitz}

We now take a (scenic) detour by considering the special case that $\graph$ is \Def{Steinitz}, i.e., $3$-connected, simple and
planar. By Steinitz's Theorem (see, for example, \cite[Chapter~4]{alma990041550820402883}),
$\graph$ is the graph of a $3$-polytope~$P$. In this case, $\graph$ has a dual graph
$\graph^\triangle$, which is the graph of the dual polytope $P^\triangle$ (assuming $\bzero$
is in the interior of $P$). This section gives a geometric interpretation of our results in later sections given the above assumptions. 

Let $C_j$ be the free $\ZZ$-module generated by the $j$-faces of $P$, with analogous modules $C_j^\triangle$ generated by the $j$-faces of $P^\triangle$.
The usual cellular chain complex given by the face structures of $P$ and $P^\triangle$ is
defined via the natural chain maps $f_j\colon C_j \to C_{ d-1-j }^\triangle$ that assign each
$j$-face of $P$ its dual face of $P^\triangle$
and the boundary maps $\partial_j : C_j \to C_{ j-1 }$ (represented as matrices over $\ZZ$),
which record the incidences among the $j$-faces and $(j-1)$-faces,
given their orientations; see, e.g., \cite[Chapter 2]{cellcomBE} for
background. 
Once the signs are chosen for $\partial_1 = \bA_G$, the remaining boundary maps are determined by $\partial_{ j-1 } \, \partial_j = 0$.
For 3-dimensional polytopes, the cellular chain complex associated to the
CW-complex given by the face-structure of $P$ and $P^{\triangle}$ is:

\begin{center}
\begin{tikzcd}[row sep=large, column sep=large]
&C_{2} \arrow{r}{\partial_2}\arrow{d}{f_{2}}
&C_{1} \arrow{r}{\partial_1}\arrow{d}{f_{1}}
&C_{0} \arrow{d}{f_{0}} \\
&C_{0}^{\triangle}   
&C_{1}^{\triangle} \arrow{l}{\partial_2^{\top} = \partial_1^{\triangle}}
&C_{2}^{\triangle} \arrow{l}{\partial_1^{\top} = \partial_2^{\triangle}}
\end{tikzcd}
\end{center}

We again can view the cographic arrangement via this diagram.
Note that $P$ has $m-n+2$ facets.


\begin{lem}\label{lem:compl}
A subset of the columns of $\partial_1^{\triangle} = \partial_2^{\top}$ is linearly independent
if and only if it induces a complement of a spanning set on $\graph$. Thus the rank of
$\partial_2^{\top} = \partial_1^{\triangle}$ is $m-n+1$.
\end{lem}

\begin{proof}
Choose a subset $S$ of the columns of $\partial_2^{\top} = \partial_1^{\triangle}$. Since $\partial_1^{\triangle}$ is the incidence matrix of $\graph^{\triangle}$, the columns of $S$ are
linearly independent if and only if they induce a forest on $\graph^{\triangle}$. On the other hand, the graph duality on $\graph$ implies that
a subset $\mathcal{T}$ of edges of $P^{\triangle}$ is a forest in
$\graph^\triangle$ if and only if the set of the edges dual to the edges of $\mathcal{T}$ is the complement of a spanning set in~$\graph$.
\end{proof}

\begin{lem}
The linear function $\partial_1^{\triangle} = \partial_2^{\top}: \RR^m \to \RR^{ m-n+2 }$ maps
the flow space $\ker \partial_1$ bijectively to the codimension-$1$ subspace of $\RR^{ m-n+2
}$ spanned by the columns of $\partial_1^{\triangle} = \partial_2^{\top}$. Thus the columns of
this matrix are normal vectors for an isomorphic copy of the cographic
arrangement living in $\RR^{ m-n+2 }$.
\end{lem}

\begin{proof}
Since $\partial_1 \, \partial_2 = 0$ respectively $\partial_2^{\top} \, \partial_1^{\top} = 0$, it follows with $ \partial_2^{\top} = \partial_1^{\triangle}$ that $\partial_1^{\triangle}\partial_1^{\top} = 0$. This means that each row of $\partial_1$ is perpendicular to each row of $\partial_1^{\triangle}$. This implies that any $\bw \in \ker(\partial_1) \cap \ker(\partial_1^{\triangle})$ stands perpendicular to all row vectors in both $\partial_1$ and $\partial_1^{\triangle}$ and thus
$\bw = \bzero$. 
Hence, $\partial_1^\triangle$ is injective when restricted to $\ker(\partial_1)$ and therefore it bijectively maps $\ker(\partial_1)$ to the column space of $\partial_1^\triangle$.
\end{proof}
Hence, in the case of Steinitz graphs (the statement is true in greater generality for planar
graphs) linear independence of normal vectors of the cographic hyperplane arrangement
$\cH(\partial_1^\triangle)$, given by  \Cref{lem:compl}, is induced in a geometric way.

Furthermore, in this case (since $\graph$ is planar), we can witness the Greene--Zaslavsky correspondence of
the regions of $\cH(\partial_1^\triangle)$ with totally cyclic orientations of $\graph$ directly:
The columns of $\partial_2^{\top} = \partial_1^{\triangle}$ are the normal vectors of the graphic
hyperplane arrangement of $\graph^{\triangle}$.
The acyclic orientations of $\graph^{\triangle}$ are in one-to-one correspondence with the
regions of the graphic arrangement of $\graph^{\triangle}$, which, in turn are in one-to-one
correspondence with the totally cyclic orientations of $\graph$ (see, for example,
\cite{Noy}; for this correspondence we do not need 3-connectivity).
 

\section{Lattice Gale zonotopes}\label{sec:latticegaledual}


We now revisit the construction of the tocyclotope of a given graph $G$: starting
with the incidence matrix of $G$, we constructed a matrix representing the cographic
matroid, from which we generated a zonotope. This process is not confined to the incidence matrix of a graph, and thus we now start with a general integral matrix $\bA\in \ZZ^{{\nn}\times \ned}$ of rank $n$. 
Recall, that we define the \Def{lattice Gale dual} of $\bA$ as the matrix $\bD\in\ZZ^{\ned
\times (\ned-\nn)}$, which consists of a $\ned-{\nn}$ \emph{lattice} basis vectors of $\ker(\bA)\cap \ZZ^\ned$ as column vectors. 
By construction, $\bD^{\top}$ represents the dual of the matroid represented by $\bA$.
If we would just choose $m-n$ basis vectors from $\ker(\bA)$, the Ehrhart polynomial of $\Zono(\bD^{\top})$ would not be unique. Therefore, we need this further restriction to a lattice basis. 

This is reminiscent of the interplay of $\Zono(\bA)$ and $\Zono(\bA^\triangle)$ (which is
combinatorially equivalent to $\Zono(\bD^{\top})$), which we
alluded to in the introduction; McMullen~\cite{mcmullernzonotopes} calls
$\Zono(\bA^\triangle)$ and $\Zono(\bD^{\top})$ \emph{derived zonotopes} and the corresponding Gale diagrams \emph{zonotopal diagrams}. He completely described the face structure of
$\Zono(\bD^{\top})$. 
Our point is to extend this description to the arithmetic
  structure (in the sense of integer points) of the derived zonotope $\Zono(\bD^{\top})$; hence our construction of the lattice Gale dual.
 
Our goal is to describe the arithmetics of the \Def{lattice Gale zonotope} $\Zono(\bD^{\top})$ in terms of $\bA$. 
While this zonotope depends on the construction of the lattice basis that yields $\bD$, \Cref{thm:galezonehrhart} is the main result in this section and shows that the arithmetic of the lattice
Gale zonotope depends only on~$\bA$.%
%

We mention the connection to arithmetic matroids:
A lattice zonotope can be seen as an instance of a representable, torsion-free arithmetic matroid with GCD property \cite{dadderio_arithmetic_2013}, where the multiplicity of the matroid is $g$ as described in \Cref{remark:mofs2}.
Then the lattice Gale dual zonotope is a representation of the dual arithmetic matroid \cite[Section 3.4]{dadderio_arithmetic_2013}.
The Ehrhart polynomial of a zonotope is a specialization of the arithmetic Tutte polynomial \cite{dadderiomoci} and the arithmetic Tutte polynomial is studied in \cite{dadderio_arithmetic_2013}.
However, to the best of our knowledge there is no explicit statement on how to compute the  arithmetic Tutte polynomial for a dual arithmetic matroid.

\begin{theorem}\label{thm:galezonehrhart}
Let $\bA\in \ZZ^{{\nn}\times \ned}$ be of rank $n$, with lattice Gale dual $\bD\in\ZZ^{\ned \times (\ned-\nn)}$.
Then we can compute the Ehrhart polynomial of the associated lattice Gale zonotope as
\begin{align}
 \ehr_{\Zono(\bD^{\top})}(t)=\sum_{S}\frac{\g(\bA_S)}{\g(\bA)}\ t^{m-\lvert S\rvert}
\end{align}
where the sum is over all spanning sets $S\subseteq[m]$ in the matroid represented by~$\bA$ and $\g(\bA)$ is defined as in  \Cref{thm:stanley}.
\end{theorem}

\begin{remark}
As (the usual) Gale duality can be used to efficiently compute the face structure of a
$d$-polytope with $k$ vertices where $k-d$ is small (but $d$ and $k$ may be large),
\Cref{thm:galezonehrhart} can be used to efficiently compute the Ehrhart polynomial of
a zonotope generated by $\bD^{\top} \in \ZZ^{(\ned-\nn) \times \ned}$ for large $m$ but small $n$:
here we have to understand only the arithmetic of the (much smaller) matrix $\bA\in
\ZZ^{{\nn}\times \ned}$.
Note that every full rank integer matrix $\bA$ can be seen as a lattice Gale dual.
Moreover,  the resulting matrix, after applying the lattice Gale dual construction twice, is unimodularly equivalent to the original matrix.
\end{remark}

\Cref{thm:cographicehrhart} is a special case of
\Cref{thm:galezonehrhart}, because the incidence matrix of a graph is totally
unimodular, and thus $\g(\bA_S) = 1$ for all $\bA_S$.
Indeed, the same reasoning implies the following specialization for any totally
unimodular matrix $\bA$,
i.e., when the associated matroid is \Def{regular}.

\begin{cor}
Let $\bA\in \ZZ^{{\nn}\times \ned}$ be a totally unimodular matrix of rank $n$, with lattice Gale dual $\bD\in\ZZ^{\ned \times (\ned-\nn)}$.
Then the Ehrhart polynomial of the associated lattice Gale zonotope is given by
\begin{align}
 \ehr_{\Zono(\bD^{\top})}(t)=\sum_{S} t^{m-\lvert S\rvert}= \sum_{k=n}^m d_k \, t^{m-k}
\end{align}
where the sum in the first sum is over all spanning sets $S$ in the regular matroid represented
by~$\bA$
and $d_k$ is the number of spanning sets of size $k$ in the matroid represented by~$\bA$.
\end{cor}

There are two main ingredients we will need to prove
\Cref{thm:galezonehrhart}. For the first we give an elementary proof here.
Given a matrix $\bA$, we denote by $\rspan(\bA)$ the real vector space spanned by its
columns and by $\zspan(\bA)$ the set of integer combination of its columns.

\begin{lem}[Gale duality for lattices]\label{lem:latticegalenew}
Let $\bA\in \ZZ^{{\nn}\times \ned}$ be of rank $n$, with a lattice Gale dual ${\bD\in\ZZ^{\ned \times (\ned-\nn)}}$.
Every choice of $k\leq m$ linearly independent rows of $\bD$ indexed by some $\compl\rho\subseteq [m]$ yields a submatrix $\bD_{\compl\rho} \in \ZZ^{ k \times (m-n) }$.
Then the complement  $\rho\coloneqq[m]\setminus\compl\rho$ defines  a matrix $\bA_\rho \in \ZZ^{ n \times (m-k) } $ consisting of the 
columns of $\bA$ indexed by $\rho$, which are spanning and hence contain a basis of $\RR^{n}$.
%
%
Then there is a bijection 
\begin{align}
 \psi\ \colon \lfrac{\ZZ^{k}}{\zspan(\bD_{\compl\rho})}\ \to\ \lfrac{\zspan(\bA)}{\zspan(\bA_{\rho})}\,.
\end{align}
\end{lem}

\begin{proof}
We may assume without loss of generality that $\rho$ indexes the first $m-k$ columns
of $\bA$ and the first $m-k$ rows of $\bD$. Let
 \begin{align}
  \psi\ \colon \lfrac{\ZZ^{k}}{\zspan(\bD_{\compl\rho})}\ &\to\  \lfrac{\zspan(\bA)}{\zspan(\bA_{\rho})}\\
 [\bv]\ &\mapsto \ [\bA_{\compl\rho} \, \bv]\,.
\end{align}
 We first show that $\psi$ is well defined and injective. Let $\bv \in \ZZ^k$. Then 
\[
  \psi [\bv] = [\bA_{\compl\rho} \, \bv] = 0 \in \, \lfrac{\zspan(\bA)}{\zspan(\bA_{\rho})}
\]
if and only if
$\bA_{\compl\rho} \, \bv = \bA_\rho \bw$ for some $\bw \in \ZZ^{ m-k }$,
that is,
\[
  \begin{pmatrix}
     -\bw\\
     \bv
  \end{pmatrix}
  \in \, \ker(\bA) \cap \ZZ^m = \zspan(\bD) \, .
\]
This means $\binom{ -\bw } \bv = \bD \bu$ for some $\bu \in \ZZ^{ m-n }$, i.e.,
$\bv = \bD_{ \compl\rho } \, \bu$, which in turn means
\[
  [\bv] = 0 \in \lfrac{\ZZ^{k}}{\zspan(\bD_{\compl\rho})} \, .
\]

To show that $\psi$ is surjective, let $\by\in\zspan(\bA)$, so
$\bA \bx = \by$ for some $\bx = \binom{ \bx_\rho }{ \bx_{ \compl\rho } } \in \ZZ^m$.
Let $\bv = \bx_{ \compl\rho } \in \ZZ^{ k }$ and $\bw = - \bx_\rho \in \ZZ^{ m-k }$. Thus
\[
  \by
  = \bA \bx
  = - \bA_\rho \bw + \bA_{ \compl\rho }  \bv \, ,
\]
i.e., $\psi[\bv] = [\by]$.
\end{proof}

The second ingredient is the quantity $g(\bF)$ from \Cref{thm:stanley}.
Recall that, for a matrix $\bF$, we defined
$g(\bF)$ as the greatest common divisor of all maximal minors of $\bF$.

\begin{remark}\label{remark:mofs2}
We recall the Smith Normal Form of a matrix $\bA \in \ZZ^{ n \times m }$ of rank
$r$, namely,
\[
  \bS \, \bA \, \bT \ = \
  \left( \begin{array}{ccccccccc}
  d_1 & \\
  & d_2 \\
  & & \ddots \\
  & & & d_r & \mbox{\hspace{.7in}} \\[.4in]
  \mbox{} 
  \end{array} \right)
\]
where $\bS \in \ZZ^{ n \times n }$ and $\bT \in \ZZ^{ m \times m }$ are invertible
matrices and $d_1 \, d_2 \cdots d_r$ equals the gcd of all $r \times r$
minors of $\bA$. Thus $\bS$ and $\bT$ are integer-lattice preserving, from which we
deduce
\begin{equation}\label{eq:smithnfconsequence}
  \left| \lfrac{ \left( \rspan(\bA) \cap \ZZ^n \right) }{ \zspan(\bA) } \right| =
d_1 \, d_2 \cdots d_r \, .
\end{equation}
Thus, if $\bA$ has full (column or row) rank,
\[
  \left| \lfrac{ \left( \rspan(\bA) \cap \ZZ^n \right) }{ \zspan(\bA) } \right| = g(\bA) \, .
\] 
Note that, by definition, $g(\bA) = g(\bA^{\top})$.
For some subset $\bF\subseteq \ZZ^n$ of
linearly independent vectors, i.e., the case of full column rank, the number $\g(\bF)$
has various interpretations:
\begin{enumerate}
 \item $\g(\bF)$ is (by definition) the greatest common divisor of all minors of size $|\bF|$ of the matrix whose columns are the elements of~$\bF$,
 \item $\g(\bF)$ is the $\lvert \bF\rvert$-dimensional relative volume of the
parallelepiped spanned by~$\bF$.
 \item 
$\g(\bF)$ is the number of cosets of the
discrete subgroup generated by $\bF$, considered as a sublattice of the integer
points in the linear span of $\bF$. 
\end{enumerate}

\end{remark}


\begin{cor}\label{cor:latticegaleimsmall}
With the same conditions and notations as in \Cref{lem:latticegalenew},
 \begin{align}
 \g(\bD_{\compl\rho})=\frac{\g(\bA_\rho)}{\g(\bA)}\,.
\end{align}
\end{cor}
\begin{proof}
First note that  
\begin{equation}
 \rspan(\bD_{\compl\rho})=\RR^k\,,
\end{equation}
since the rows of $ \bD_{\compl\rho}$ are linearly independent and $k\leq m-n$.
By \Cref{remark:mofs2} and \Cref{lem:latticegalenew},
\[
 \g(\bD_{\compl\rho})=\left| \lfrac{\ZZ^{k}}{\zspan(\bD_{\compl\rho})}\right| 
    = \left| \lfrac{\zspan(\bA)}{\zspan(\bA_\rho)}\right| = \frac{ \left| \sfrac{\ZZ^n}{\zspan(\bA_\rho)}\right|}{ \left| \sfrac{\ZZ^n}{\zspan(\bA)}\right| }
    = \frac{\g(\bA_\rho)}{\g(\bA)}\,. \qedhere
\]
\end{proof}

\begin{proof}[Proof of \Cref{thm:galezonehrhart}]
By Stanley's \Cref{thm:stanley},
 \begin{align}
  \ehr_{\Zono(\bD^{\top})}(t)=\sum_{J}\g \left( (\bD^{\top})_J \right) t^{\lvert J\rvert}
 \end{align}
where $J$ indexes linearly independent subsets of columns of $\bD^{\top}$, i.e.,
 the sum is over independent sets $J\subseteq [m]$ in the dual matroid.
By matroid duality, these sets correspond to spanning sets $S=[m]\setminus J$ in the primal matroid.
By \Cref{remark:mofs2}, $\g ( (\bD^{\top})_J ) = \g(\bD_J)$, and from \Cref{cor:latticegaleimsmall}
we know that $\g(\bD_J)=\frac{\g(\bA_S)}{\g(\bA)}$,
and so
\begin{align}
 \ehr_{\Zono(\bD^{\top})}(t)=\sum_{S}\frac{\g(\bA_S)}{\g(\bA)}\ t^{m-\lvert S\rvert}
\end{align}
where the sum is now over spanning sets in the primal matroid.
\end{proof}

\section{Tocyclotopes for signed graphs}\label{sec:signedtoyclotopes}

The goal of this section is to construct the signed tocyclotope and then compute its Ehrhart polynomial in terms of signed graph-theoretic data. 
We could define the signed tocyclotope as a lattice Gale dual of the signed acyclotope as explained in the previous section. 
However, there is a more combinatorial and concrete route via bidirected network matrices, which can be seen in analogy to the graphic case. 
This construction is  in fact a special case of our more general framework in \Cref{sec:latticegaledual}, which will help us with the computation of the Ehrhart coefficients for the tocyclotopes.

As mentioned in \Cref{ssec:signedgraphs}, oriented signed graphs are equivalent to bidirected graphs. 
Bidirected graphs were first defined by Edmonds and Johnson
\cite{jack_edmonds_introduction_1967,jack_edmonds_matching_1970}.
Appa and Kotnyek  \cite{appakotnyek} studied a bidirectional  analog of network matrices. 
One of their central results, which we will use, is conditions on when the duals of those matrices is integral. 
In general, those inverses are half-integral, a result that first appeared in \cite{bouchet_nowhere-zero_1983}.
For more information see also the references in \cite{bolkerzaslavsky, bouchet_nowhere-zero_1983, kotnyek_generalization_2002}.
%
%
%
%
%
%

From now on we want to assume that our (signed) graph is connected. 
If it is not connected, we can apply the following results to each of the connected
components and then take the appropriate product;
we will give more details in \Cref{wlogproducts} below.
Additionally, we will assume that the  incidence matrix $\bA_\sgraph \in \ZZ^{{\nn}\times \ned}$ has
full rank, i.e., $\rank(\bA_\sgraph)=\nn$.
If this is not the case then the signed graph is balanced and can
hence be considered as an unsigned graph. 

\subsection{Binet matrices and the tocyclotope}\label{sec:signedcoacyclotope}

We will give the definition of a binet matrix%
\footnote{\hspace{1pt} ``The term binet is used here as a short form for \emph{bi}directed \emph{net}work,
but by coincidence it also matches the name of Jacques Binet (1786--1856) who worked on the foundations of matrix theory and gave the rule of matrix multiplication.''  \cite[page 46]{kotnyek_generalization_2002}}
as it was introduced by Appa and Kotnyek \cite{appakotnyek} in order to generalize the dual of network matrices to bidirected graphs.%
\footnote{Note that the network matrix is the reduced incidence matrix of a graph, while the binet matrix is a part of the dual matrix in the signed graphic case.}

Let $\bA_\sgraph\in\ZZ^{\nn\times\ned}$ be the incidence matrix of the signed graph $\sgraph$
and let $T\subseteq E$ be a subset of the edges of $\sgraph$ that forms a basis, as
discussed in \Cref{thm:signedindependence}.
This implies that the submatrix $\bT\in\ZZ^{\nn\times\nn}$ of $\bA_\sgraph$ formed by choosing the columns indexed by $T$ is invertible over $\RR$.
After reordering columns, we can write the incidence matrix $\bA_\sgraph$ as $[\bR \, | \, \bT]$, where $\bR\in\ZZ^{\nn\times(\ned-\nn)}$ is the matrix formed from columns indexed by $R\coloneqq E\setminus T$.
Then we multiply $\bA_\sgraph=[\bR \, | \, \bT]$ with $\bT^{-1}$ from the left to obtain 
\begin{equation}
 \bT^{-1}\bA_\sgraph 
 =[\bT^{-1}\bR \, | \, \bI] = [\bB \, | \, \bI]\,,
\end{equation}
where $\bI\in\ZZ^{\nn\times\nn}$ is the unit matrix and $\bB\coloneqq \bT^{-1}\bR\in\RR^{\nn\times(\ned-\nn)}$.
The matrix $\bB$ is called the \Def{binet matrix}. %
%
 Appa and Kotnyek further present a graphical algorithm to compute binet matrices \cite{appakotnyek}. 
The algorithm gives an easier and more direct way of computing binet matrices; 
see also~\cite{bolkerzaslavsky}.
%
\begin{lem}[{\cite[Lemma 17]{appakotnyek}}]\label{lem:integral} 
 Let $\sgraph$ be a signed graph and $T\subseteq E$ be a subset that forms a maximal pseudo-forest.
 The binet matrix $\bB=\bT^{-1}\bR$ is integral if and only if 
 one of the following conditions holds:
 \begin{enumerate}[(a)]
  \item\label{lem:integral:itm:halfedge} every connected component in the maximal pseudo-forest formed by $T$ is a (signed) halfedge-tree, or
  \item\label{lem:integral:itm:connected}  $\sgraph$ does not contain halfedges and $T$ has one connected component.
 \end{enumerate}
\end{lem}
Since we assumed the signed graph $\sgraph$ to be connected, we can always choose a pseudo-tree $T\subseteq E$ that fulfills one of the conditions in \Cref{lem:integral}:
If the signed graph contains halfedges, choose a connected basis that contains one of the halfedges (case \ref{lem:integral:itm:halfedge} in \Cref{lem:integral}),
otherwise choose any other connected basis (case \ref{lem:integral:itm:connected} in \Cref{lem:integral}).
Then we know that the matrix $[\bB\, |\, \bI]\in \ZZ^{n\times m}$ has integral coefficients.

It is immediate from the construction that the rows of $\bD^{\top}\coloneqq[\bI\,|-\bB^{\top}]\in \ZZ^{(m-n)\times m} $,
where here $\bI\in\ZZ^{(\ned-\nn)\times(\ned-\nn)}$ and $-\bB^{\top}=-(\bT^{-1}\bR)^{\top}\in\ZZ^{(\ned-\nn)\times\nn}$,
are contained in the kernel of $\bA_\sgraph$.
Since the matrix has full rank $m-n$, its rows span the kernel of $\bA_\sgraph$.
Note that $\g([\bI\,|-\bB^{\top}])=1$ because the maximal minor given by the identity matrix $\bI$ equals one and hence the greatest common divisor of all minors as well.
From \Cref{remark:mofs2} it follows that the rows of $[\bI\,|-\bB^{\top}]$ form a lattice basis for $\ker(\bA_\sgraph)\cap \ZZ^m$. 
So $\bD^{\top}=[\bI\,|-\bB^{\top}]\in \ZZ^{(m-n)\times m} $ is the transpose of a lattice Gale dual of $\bA_\sgraph$.
Therefore this combinatorial construction fits into the general framework from \Cref{sec:latticegaledual}.

Hence, we define the \Def{tocyclotope for signed graphs} as the zonotope ${\Zono([\bI|-\bB^{\top}])}$. 
As in the case of unsigned graphs, this zonotope depends on our choice of~$\bT$; however, not only is its
face structure independent of this choice (by \cite{mcmullernzonotopes}), the same is true for its
Ehrhart polynomial. 
This follows from \Cref{thm:galezonehrhart} or from the observation that for every right choice of $\bT$ the rows of the resulting matrix $[\bI\,|-\bB^{\top}]$ form a lattice basis for $\ker(\bA_\sgraph)\cap \ZZ^m$ and therefore are  unimodular equivalent.


Parallel to the definitions for graphs, the \Def{flow space} of the signed graph
$\sgraph$ is $\ker(\bA_\sgraph)$, and the \Def{signed cographic arrangement} is
the hyperplane arrangement induced by the coordinate hyperplanes of $\RR^m$ on
$\ker(\bA_\sgraph)$; see, e.g., \cite{nnz,chenwang,chen_resolution_2017}.
The proof of the following lemma is almost verbatim that of \Cref{lem:flow_space}.

\begin{lem}\label{lem:signed_flow_space}
Let $\sgraph$ be a signed graph without coloops and
let $ \bD^{\top}=[\bI|-\bB^{\top}]$ be as described above.
The linear surjection $\bD^{\top}: \RR^m \to \RR^{ m-n }$ maps the flow space
$\ker(\bA_\sgraph)$ bijectively to $\RR^{ m-n }$. Thus the columns of the matrix $\bD^{\top}$ are normal vectors for an isomorphic copy of the signed cographic arrangement living in~$\RR^{ m-n}$.
\end{lem}

Although this is not a main theme of this paper, we add a remark about the face structure
of the tocyclotope, as it follows directly from (oriented matroid) duality.
We recall that a cycle is a minimally dependent set of edges that is oriented in such a way
that it has neither a sink nor a source.
Recall that an orientation is totally cyclic if every (bioriented) edge is contained in a cycle.
The regions of the signed cographic arrangement, and therefore the vertices of the signed tocyclotope, correspond bijectively to totally cyclic orientations of the signed graph $\sgraph$; see \cite[proof of Theorem 4.5.(b)]{nnz}.
Higher dimensional faces of the signed tocyclotope can be understood via the flats of the signed cographic arrangement.

\subsection{The Ehrhart polynomial of the tocyclotope}\label{ssec:ehrtocyclotope}

The goal of this section is to prove the combinatorial description of the coefficients in the Ehrhart polynomial of signed tocyclotopes.

\begin{theorem}\label{thm:signedtocyclotopeehr} 
Let $\sgraph$ be a connected signed graph whose incidence matrix has full rank. 
Choose a connected basis $T\subseteq E$ that contains a halfedge if $\sgraph$ contains a halfedge.
Then the Ehrhart polynomial of the tocyclotope $\Zono([\bI \, |-\bB^{\top}])\subseteq\RR^{\ned-\nn}$ is
\begin{equation}
 \ehr_{\Zono([\bI \, |-\bB^{\top}])}(t)= \begin{cases}
                  \sum_{S} 2^{\mplc(S)} t^{\ned-\lvert S\rvert} & \text{if $\sgraph$ contains a halfedge,}\\
                  \sum_{S} 2^{\mplc(S)-1} t^{\ned-\lvert S\rvert} & \text{if $\sgraph$ does not contain a halfedge,}
                 \end{cases}
                 \label{eq:tocyclotopeehrhart}
\end{equation}
where the sums run over all sets $S\subseteq E$ that contain a basis of $\sgraph$, i.e., $\subsgraph{S}$ contains a maximal pseudo-forest of $\sgraph$, and 
\begin{equation}
 \mplc(S)\coloneqq \min_{\substack{\tilde{T}\subseteq S}}
 \left(\pc(\tilde{T})+\lc(\tilde{T})\right)\,,
\end{equation}
with the minimum taken over all maximal pseudo-forests $\tilde{T}$ in $\sgraph$ contained in the spanning set $S$.
\end{theorem}
We will again apply Stanley's \Cref{thm:stanley} for zonotopes. 
For that we need a combinatorial understanding of $\g(\bJ)$, where the columns in the  submatrix
$\bJ$ of $\bD^{\top}=[\bI \, |-(\bT^{-1}\bR)^{\top}]$ are linearly independent.
Recall that they correspond to independent sets $J$ in the dual signed graphic matroid.
Hence they correspond to subsets of edges $S=E\setminus J$ in the signed graph $\sgraph$ that contain a basis, i.e., a maximal pseudo-forest.
\begin{cor}
 A subset of columns $\bJ$ of $\bD^{\top}=[\bI \, |-(\bT^{-1}\bR)^{\top}]$ (as constructed above) is  linearly independent if and only if the subset $\bS$ of columns in $\bA_\sgraph$ indexed by $S=E\setminus J$ is a spanning set, i.e., $S$ contains a maximal pseudo-tree in $\sgraph$.
 In this case,
 \begin{align}
  \g(\bJ)= \frac{\g(\bS)}{\g(\bA_\sgraph)}\,.
 \end{align}
\end{cor}
\begin{proof}
 This follows from \Cref{cor:latticegaleimsmall}.
\end{proof}
Therefore, it remains to understand the parameter $\g(\bS)$ for spanning sets in the signed graph.

%
\begin{lem}\label{lem:gofSforsignedgraphs}
Let $S\subseteq E$ be a spanning set. Then there exists a maximal forest $F\subseteq S$ such that
 $\g(\bA_{\subsgraph{S}})=\g(\bA_{\subsgraph{F}})$.
 Moreover, this maximal forest $F$ will be one with a minimal number of pseudo-tree components plus loop-tree components.
\end{lem}
\begin{proof}
From \Cref{remark:mofs2} we know that $\g(\bA_{\subsgraph{S}})$ is the greatest common divisor of all minors of size $n$ in $\bS$. 
Since all minors are powers of $2$ (by \Cref{lem:relvolpi}), the greatest common divisor is the lowest power of $2$ that appears.
The selection of columns in $\bS$ for which the minor attains its minimum corresponds to a forest $F\subseteq S$ of the kind that we are looking for. Then
\begin{align}
  \g(\bA_{\subsgraph{S}})=\g(\bA_{\subsgraph{F}})= 2^{\pc(F)+\lc(F)} = 2^{\mplc(S)}\, . 
\end{align}
\end{proof}

\begin{proof}[Proof of \Cref{thm:signedtocyclotopeehr}]

By \Cref{thm:galezonehrhart} 
\begin{equation}
 \ehr_{\tocyc(\sgraph)}(t)=\sum_{S}\frac{\g(\bA_S)}{\g(\bA_\sgraph)}\ t^{m-\lvert S\rvert}
\end{equation}
where the sum is over all spanning sets $S$ in the matroid represented by~$\bA_\sgraph$,
i.e., over all subsets $S\subseteq E$ that contain a maximal pseudo-forest of $\sgraph$.

Note that for connected signed graphs $\sgraph$ of full rank, $\zspan(\bA_\sgraph)=\ZZ^\nn$ 
(and hence $\g(\bA_\sgraph)=1$ by \Cref{remark:mofs2}) if and only if $\sgraph$ contains a halfedge by \Cref{lem:relvolpi}.
In the case of connected signed graphs without halfedges we can apply \Cref{cor:latticegaleimsmall}, and we will get a correction factor of $2$ since then  $\g(\bA_\sgraph)=2$  again by \Cref{lem:relvolpi}.
This explains the case distinction in \eqref{eq:tocyclotopeehrhart} and the difference of
a factor of 2 between the cases.

The last missing piece now is to understand $g(\bA_S)=g(\bS)$. This is given in \Cref{lem:gofSforsignedgraphs}: we need to find the minimal possible number $\mplc(S)$ of loop-tree components plus pseudo-tree components in a maximal pseudo forest in the spanning set $S$.
Then we arrive at
\[
 \ehr_{\tocyc(\sgraph)}(t)= \begin{cases}
                  \sum_{S} 2^{\mplc(S)} t^{\ned-\lvert S\rvert} & \text{if $\sgraph$ contains a halfedge,}\\
                  \sum_{S} 2^{\mplc(S)-1} t^{\ned-\lvert S\rvert} & \text{if
$\sgraph$ does not contain any halfedges.} 
                 \end{cases}
\]
\end{proof}

We conclude this section with the extension of the results to signed graphs that are not connected.

\begin{remark}\label{wlogproducts}
Let $\sgraph$ be an arbitrary signed graph with connected components $\sgraph_1, \dots, \sgraph_c$. 
Then we can order nodes and edges so that the incidence matrix $\bA_\sgraph$ has a block structure given by the connected components:
 \begin{align}
  \bA_\sgraph = \begin{bmatrix}
                 \bA_{\sgraph_1} & \bzero & \bzero & \dots &\bzero\\
                 \bzero & \bA_{\sgraph_2} &\bzero & \dots &\bzero\\
                 \vdots & \ddots & \ddots & \ddots & \vdots \\
                 \vdots &  & \ddots & \ddots & \bzero \\
                 \bzero & \dots & \dots &\bzero &\bA_{\sgraph_c}
                \end{bmatrix} .
 \end{align}
 This implies that the acyclotope of the signed graph $\sgraph$ is simply the Cartesian product of the acyclotopes of the connected components:
 \begin{align}
 \Zono(\bA_\sgraph)= \Zono(\bA_{\sgraph_1})\times \dots \times\Zono(\bA_{\sgraph_c})\,.
 \end{align} 
 Hence the Ehrhart polynomial of $\Zono(\bA_\sgraph)$ is a product of Ehrhart polynomials 
  \begin{align}
  \ehr_{\Zono(\bA_\sgraph)}= \ehr_{\Zono(\bA_{\sgraph_1})}\cdots \ehr_{\Zono(\bA_{\sgraph_c})}\,.
  \end{align}
  A similar decomposition property can be found on the level of matroids.
  Here the signed graphical matroid $\matroid(\sgraph)$ is the direct sum
  \begin{align}
  \matroid(\sgraph)=\matroid(\sgraph_1)\oplus\dots\oplus\matroid(\sgraph_c)\,.
  \end{align}
  This structure is preserved under taking matroid duals, hence
  \begin{align}
  \dualmatroid(\sgraph)=\dualmatroid(\sgraph_1)\oplus\dots\oplus\dualmatroid(\sgraph_c)\,.
  \end{align}
  So we can also apply our duality construction block by block to achieve a dual representation 
 \begin{align}
  \bD_\sgraph = \begin{bmatrix}
                 \bD_{1} & \bzero & \dots &\bzero\\
                 \bzero & \ddots & \ddots & \vdots \\
                 \vdots  & \ddots & \ddots & \bzero \\
                 \bzero  & \dots &\bzero &\bD_{c}
                \end{bmatrix}\,.
 \end{align}
 Then  the signed tocyclotope is the Cartesian product
  $\Zono(\bD_\sgraph)= \Zono(\bD_{1})\times \dots \times\Zono(\bD_{c})$,
  and hence its Ehrhart polynomial is again a product of Ehrhart polynomials
  \begin{align}
   \ehr_{\Zono(\bD_\sgraph)}= \ehr_{\Zono(\bD_{1})}\cdots \ehr_{\Zono(\bD_{c})}\,.
  \end{align}

\end{remark}


We conclude with a concrete open questions.
Recall that the lattice points in the acyclotope (for unsigned graphs) arise as indegree vectors from all orientations of the graph.
 While this correspondence is bijective for acyclic orientations, it is not for general orientations.
 For tocyclotopes we know that the vertices correspond to totally cyclic orientations.
 Is there a similar interpretation for all lattice points in the tocyclotope?
 One way to address this question might be via the algorithm in~\cite{appakotnyek}.

\bibliographystyle{amsplain}
\bibliography{bib.bib}

\setlength{\parskip}{0cm} 

\end{document}